\documentclass[reqno, 11pt]{amsart}
\usepackage{pdfsync}
\usepackage[colorlinks, linkcolor=blue, citecolor=red]{hyperref}
\usepackage{url}
\usepackage{times}

\newtheorem{theorem}{Theorem}[section]
\newtheorem{lemma}[theorem]{Lemma}

\theoremstyle{definition}
\newtheorem{definition}[theorem]{Definition}
\newtheorem{ex}[theorem]{Example}

\newtheorem{remark}[theorem]{Remark}

\numberwithin{equation}{section}

\usepackage{bbm}
\usepackage{url}
\usepackage{euscript}
\usepackage{pb-diagram}
\usepackage{lamsarrow}
\usepackage{pb-lams}
\usepackage{amsmath}
\usepackage{amsthm}

\usepackage{amsxtra}
\usepackage{amssymb}
\usepackage{pifont}
\usepackage{amsbsy}

\usepackage{graphicx}
\usepackage{epstopdf}

\newskip\aline \newskip\halfaline
\def\skipaline{\vskip\aline}

\def\qedbox{$\rlap{$\sqcap$}\sqcup$}
\def\qed{\nobreak\hfill\penalty250 \hbox{}\nobreak\hfill\qedbox\skipaline}

\newcommand\bR{{\mathbb R}}

\newcommand\bZ{{\mathbb Z}}

\DeclareMathOperator{\supp}{{\rm supp}}

\DeclareMathOperator{\codim}{codim}

\newcommand{\bm}{\boldsymbol{m}}
\newcommand{\bn}{\boldsymbol{n}}

\newcommand{\bsC}{\boldsymbol{C}}

\newcommand{\bsI}{\boldsymbol{I}}

\newcommand{\bsN}{\boldsymbol{N}}
\newcommand{\bsR}{\boldsymbol{R}}
\newcommand{\bsS}{\boldsymbol{S}}
\newcommand{\bsT}{\boldsymbol{T}}
\newcommand{\bsU}{{\boldsymbol{U}}}
\newcommand{\bsV}{{\boldsymbol{V}}}

\newcommand{\bgamma}{\boldsymbol{\gamma}}

\newcommand{\bom}{\boldsymbol{\omega}}

\newcommand{\si}{{\sigma}}

\newcommand{\ve}{{\varepsilon}}

\newcommand{\vfi}{{\varphi}}

\newcommand{\eA}{\EuScript{A}}
\newcommand{\eB}{\EuScript{B}}
\newcommand{\eC}{\EuScript{C}}

\newcommand{\eF}{\EuScript{F}}

\newcommand{\h}{\EuScript H}
\newcommand{\eH}{\EuScript H}
\newcommand{\eI}{\EuScript{I}}

\newcommand{\eK}{\EuScript{K}}
\newcommand{\eL}{\EuScript{L}}

\newcommand{\eN}{\EuScript{N}}

\newcommand{\eP}{\EuScript{P}}

\newcommand{\eR}{\EuScript{R}}
\newcommand{\eS}{\EuScript{S}}

\newcommand{\eX}{\EuScript{X}}

\newcommand{\eZ}{\EuScript{Z}}

\newcommand{\ra}{\rightarrow}
\newcommand{\hra}{\hookrightarrow}

\newcommand{\lan}{\langle}
\newcommand{\ran}{\rangle}

\def\inpr{\mathbin{\hbox to 6pt{\vrule height0.4pt width5pt depth0pt \kern-.4pt \vrule height6pt width0.4pt depth0pt\hss}}}

\newcommand{\pa}{\partial}

\newcommand{\dxi}{{\dot{\xi}}}
\newcommand{\dx}{{\dot{x}}}

\newcommand{\ori}{\boldsymbol{or}}

\DeclareMathOperator{\cl}{\boldsymbol{cl}}

\DeclareMathOperator{\Cr}{\mathbf{Cr}}

\DeclareMathOperator{\Graff}{\mathbf{Graff}}

\DeclareMathOperator{\limi}{\underrightarrow{\mathrm{lim}}}

\newcommand{\dual}{{\spcheck{}}}

\newcommand{\chit}{\chi_{\mathrm{top}}}
\newcommand{\chio}{\chi_{\mathrm{o}}}

\begin{document}

\title{On the normal cycles of   subanalytic sets}

\date{Started March 4, 2010. Completed  on March 5, 2010.
Last modified on {\today}. }

\author{Liviu I. Nicolaescu}

\address{Department of Mathematics, University of Notre Dame, Notre Dame, IN 46556-4618.}
\email{nicolaescu.1@nd.edu}
\urladdr{\url{http://www.nd.edu/~lnicolae/}}

\begin{abstract}  We present a very short complete  proof of the existence of the normal cycle of a  subanalytic set. The approach is  Morse theoretic in flavor   and   relies heavily on    techniques   from  $o$-minimal topology.    \end{abstract}

\maketitle

\tableofcontents

\section{Introduction}
\setcounter{equation}{0}

The normal and conormal cycles of a reasonably well behaved subset $X$ of an oriented Euclidean space  $\bsV$ of dimension $n$ are  currents  that  encode rather subtle topological  and  geometric features of  the set.    The normal cycle $\bsN^X$ is a Legendrian  cycle  contained in the unit sphere bundle  $S(T\bsV)$ associated to the tangent bundle $T\bsV$, while  the conormal cycle is   a Lagrangian cycle $\bsS^X$  in the cotangent  bundle $T^*\bsV$.     The  two objects completely determine each other in a   canonical fashion.   Their precise  definitions are rather sophisticated in general, but they can be easily     described in many concrete examples.

For example, if  $X$ is a   submanifold of $\bsV$, then  $\bsN^X$  can be identified with the integration current defined by the total space of the unit sphere bundle associated to the normal bundle of the embedding $X\hra \bsV$, while $\bsS^X$ can be identified with the current of integration defined by the total space of the conormal bundle of the embedding $X\hra \bsV$.

The  normal   cycle  $\bsN^X$  is intimately related to H. Weyl's celebrated tube formula \cite{Weyl}.  More precisely,  there exist canonical forms (see \cite[\S 0.3]{Fu2})  $\eta_0,\dotsc, \eta_{n-1}\in \Omega^{n-1}\bigl(\, S(T\bsV)\,\bigr)$ such that,   for any compact submanifold $X\hra\bsV$, the integrals
\begin{equation}
\mu_k(X):=\int_{\bsN^X}\eta_k
\label{eq: n1}
\end{equation}
can be  expressed as integrals over $X$ of universal   polynomials in the curvature of the induced metric on $X$.     For example,  if $m=\dim X$, then $\mu_m(X)$ is the $m$-dimensional volume of $X$, and $\mu_{m-2}(X)$   coincides (up to a universal  multiplicative constant) with the integral over $X$ of the scalar curvature.  The quantity $\mu_0(X)$ is the Euler characteristic  of $X$  which, according to the Gauss-Bonnet theorem, can be expressed  as the integral of a universal polynomial  in the curvature of of $X$. The quantities  $\mu_k(X)$ are known  as \emph{curvature measures}.   They are the  key ingredients in the tube formula that  states that for any sufficiently small $r>0$ the volume of a tube of radius $r$ around a  compact submanifold $X\hra \bsV$ of dimension $m$ is (see \cite[\S 9.3.3.]{N0})
\begin{equation}
V_X(r)=\sum_{k=0}^m \mu_{m-k}(X) \bom_{n-m+k}r^{n-m+k},
\label{eq: tube}
\end{equation}
where $\bom_p$ denotes the volume of the unit $p$-dimensional ball. 

If $X$ happens to be a bounded domain in $\bsV$ with sufficiently regular boundary $\pa X$, then  we   have a unit outer normal vector field 
\begin{equation}
\bn:  \pa X\ra  S(\bsV):=\bigl\{ v\in\bsV:\;\;|v|=1\,\bigr\}
\label{eq: g1}
\end{equation}
and the   normal cycle  $\bsN^X$  the integration current defined by  the graph of the above map; see Example \ref{ex: reg}.    In this case the integrals $\int_{\bsN^X}\eta_k$  can be expressed  as integrals over $\pa X$ of universal polynomials in the     second fundamental form of  the hypersurface $\pa X$, and they are involved  in a  tube formula similar to (\ref{eq: tube}), \cite[\S 9.3.5]{N0}. If additionally $X$ happens to be convex, then  the curvature measures $\mu_k(X)$ coincide with     the  \emph{Quermassintegrale}  constructed  by  H. Minkowski,  \cite{KlR}.             

In the  groundbreaking work  \cite{Feder0}     H. Federer   has  explained how to  associate curvature measures to subsets  of $\bsV$  of positive reach.        This class of subsets contain as subclasses the smooth submanifolds  of $\bsV$, the  bounded   domains  with smooth boundary and the convex bodies in $\bsV$, and in these cases Federer's curvature measures specialize to the curvature measures described above.

As explained in \cite{KlR},  the \emph{Quermasseintegrale}  can  be extended in a canonical fashion to   finitely additive measures (valuations)  defined  on the collection of polyconvex subsets of $\bsV$, i.e.,    sets that are finite unions of  convex bodies.  In particular, the quantity $\mu_k(X)$ is well defined  for any compact $PL$ subset of $\bsV$.   For most  $PL$ sets the Gauss map (\ref{eq: g1}) is not defined and the above definition of $\bsN^X$ is meaningless. 

In the  mid 1980s, J. Cheeger, S. M\"{u}ller and R .Schr{a}der \cite{CMS} have associated to  an arbitrary  compact $PL$ subset $X\subset \bsV$ a   Legendrian cycle  $N^X$  contained in $S(T\bsV)$ such that 
\[
\mu_k(X)=\int_{\bsN^X} \eta_k
\]
Their   elementary construction  is very intuitive and  is based on elementary    Morse theory on $PL$-spaces.    Moreover, the correspondence $X\mapsto \bsN^X$ from the collection of  compact $PL$ subsets of $\bsV$ to the Abelian group of  Legendrian cycles in $S(T\bsV)$  is a finitely additive measure, i.e.,
\begin{equation}
\bsN_{X\cup Y}=\bsN^X+\bsN^Y-\bsN^{X\cap Y},
\label{eq: inc-exc}
\end{equation}
for any $PL$ sets $X,Y$.   The paper  \cite{CMS}  includes the first formal definition of the normal cycle. Roughly speaking, the normal cycle  $\bsN^X$  is  designed  to be an  ingenious catalogue  of the Morse theoretic behavior of the restrictions to $X$ of ``typical'' linear functions on $\bsV$.

A few  years after   \cite{CMS},   M. Kashiwara and P. Schapira \cite{KaSch} have shown how to associate  a   normal cycle to any bounded subanalytic  subset of $\bsV$.       Although the Morse theoretic point of view is still in the background, their  approach is sheaf theoretic  and geared towards topological applications.  Their proof is quite sophisticated as it relies on  highly nontrivial results about the derived categories of  sheaves.       J. Sch\"{u}rmann \cite{Schu} has  proposed a simpler sheaf theoretic  construction of the normal cycle, but this too requires  a good  familiarity    with  stratified spaces  the basic operations in the derived category of sheaves.

  Almost immediately following the work of Kashiwara and Schapira, J. Fu   \cite{Fu2}   gave another  construction  of the normal   cycle   of a subanalytic set using    methods of geometric measure theory.    His proof  is technically very demanding, and  the complete details are spread over several papers.
  
  Very recently, A. Berning \cite{Ber}  has  proposed  a very ingenious and elegant  elementary construction of the  normal cycle of a subanalytic set using  the recent advances in $o$-minimal topology and  basic facts about currents. Unfortunately there  is a  flaw in a key existence result, \cite[Lemma 6.4]{Ber}; see Remark \ref{rem: glitch}(a) for more detail;s. The present paper grew out of  our attempts to  fix that flaw.

 The  main goal  of this paper is to  describe  a very short  complete proof of the existence of the normal cycle  of a subanalytic set   by relying on techniques and ideas from $o$-minimal topology.    We rely on several fundamental facts of geometric measure theory that we  survey in Appendix \ref{s: c}. The    construction  has a Morse theoretic  flavor,  and it   is based  on two key principles.
  
  \begin{itemize}
  
  \item  \emph{An uniqueness   result   closely related   to the uniqueness results of  J. Fu \cite[Thm. 3.2]{Fu2}.}        Loosely speaking, this uniqueness result  states that there exists   a unique  Legendrian cycle in $\Sigma\dual\times \bsV$ that  catalogs  in   a certain explicit fashion (see Remark \ref{rem: normc}(a)) the  Morse theoretic properties of the restrictions to $X$ of   generic linear functions on $\bsV$.  When it exists, this unique cycle is the  normal cycle $\bsN^X$ of $X$.

  \item \emph{An approximation process   pioneered by  J. Fu \cite{Fu2}.}  More precisely, we show that for any    compact subanalytic set $X$ we can find  a    family   of  bounded domains  $(X_\ve)_{\ve >0}$ with  $C^3$-boundaries such that $X=\cap_{\ve>0}X_\ve$ and the normal cycles $\bsN^{X_\ve}$ converge in the sense of currents  to a current satisfying the requirements of the uniqueness theorem. Thus, the limit cycle must be the normal cycle of $X$.

  \end{itemize}

The  resulting correspondence $X\mapsto\bsN^X$, $X$ bounded subanalytic set,   satisfies the inclusion-exclusion principle  (\ref{eq: inc-exc}) and for $PL$ sets  it coincides  with the normal cycle   constructed in \cite{CMS}. The concrete  implementation of our  strategy   was possible only due   to the    recent advances in $o$-minimal   topology.   Here is a more technical description of our main results.  

Let $\bsV$ be an oriented real Euclidean  vector space of dimension  $n$. Denote by $\bsV\dual$ its  dual,  and by $\Sigma\dual$ the unit sphere in $\bsV\dual$.  We identify  the cotangent bundle $T^*\bsV$ with the product $\bsV\dual\times \bsV$.    We have two canonical projections
 \[
 p: \bsV\dual\times\bsV\ra \bsV\dual,\;\;\pi: \bsV\dual\times\bsV\ra \bsV.
 \]
Let $\lan -,-\ran: \bsV\dual\times \bsV\ra   \bR$ denote   the canonical pairing
 \[
 \bsV\dual\times \bsV\ni (\xi, x)\mapsto  \lan\xi, x\ran:= \xi(x)\in\bR.
 \]
The Euclidean  metric $(-,-)$ on $\bsV$  defines isometries (the classical lowering/raising the indices operations)
 \[
 \bsV\ni x\mapsto x_\dag\in \bsV\dual,\;\;\bsV\dual\ni \xi\mapsto\xi^\dag\in\bsV,
 \]
 \[
 \lan x_\dag,  y\ran =(x,y),\;\; \lan \xi, y\ran =(\xi^\dag, y),\;\;\forall x,y\in\bsV,\;\;\xi\in\bsV\dual.
 \]
  Let $\alpha\in\Omega^1(T^*\bsV)$ denote the canonical $1$-form on the cotangent bundle.  More explicitly, if $x^1,\dotsc, x^n$
 are Euclidean  coordinates on $\bsV$, and $\xi_1,\dotsc, \xi_n$ denote the induced Euclidean coordinates on $\bsV\dual$, then
 \[
 \alpha=\sum_i \xi_i dx^i.
 \]
We denote by $\omega\in \Omega^2(T^*\bsV)$ the associated symplectic form
\[
\omega=-d\alpha=\sum_i dx^i\wedge d\xi_i.
\]
In this paper we will work extensively with  subanalytic  objects.     Our subanalytic  sets are the  sets  in the $o$-minimal structure  ${\bR}_{\mathrm{an}}$ as defined  Appendix \ref{s: b}  to which we refer  for more details. 

We will work with special classes of currents.  For the reader's convenience we have gathered in Appendix \ref{s: c}    the basic  notations  and facts  involving currents that we use throughout this paper. For any closed subanalytic subset $X\subset \bsV\dual\times \bsV$ we denote  by $\eC_k(X)$   the Abelian group of  subanalytic, $k$-dimensional currents  with support on $X$; see  Appendix \ref{s: c} for details.    More precisely, $\eC_k(X)$ is the Abelian subgroup of $\Omega_k(\bsV\dual\times \bsV)$ spanned by the currents of integration over oriented  $k$-dimensional  subanalytic submanifolds contained in $X$.  If $S\in \eC_k(\Sigma\dual\times \bsV)$, and $\xi\in \Sigma\dual$, we denote by $S_\xi$ the $p$-slice of $S$ over $\xi$,
 \[
 S_\xi:=\lan S, p, \xi\ran\in \eC_{k-\dim\Sigma\dual}(\Sigma\dual\times \bsV)
 \]
 As explained in Appendix \ref{s: c}, the slice  $S_\xi$  exists for all $\xi$ outside a codimension $1$ subanalytic subset of $\Sigma\dual$ and it is supported on the fiber $p^{-1}(\xi)\cap\supp S$.    More precisely, $S_\xi$ is well defined if the  fiber $p^{-1}(\xi)\cap \supp S$ has the expected  dimension, $\dim S-\dim \Sigma\dual$.
  
   If $S$   is  the current of integration along an oriented $k$-dimensional  manifold, then for generic $\xi$ the slice $S_\xi$ is the current of integration along the fiber $S\cap p^{-1}(\xi)$ equipped with a canonical orientation.  In general, the slice gives a precise meaning   as a current to the intersection of $S$ with the fiber $p^{-1}(\xi)$, provided that this intersection has the ``correct'' dimension.

If $X\subset \bsV$ is a compact subanalytic set,  $\xi\in \Sigma\dual$,   and $x\in X$ we set
\[
X_{\xi>\xi(x)}:=\bigl\{  y\in X;\;\;\xi(y)>\xi(x)\,\bigr\},\;\; i_X(\xi, x):=1-\lim_{r\searrow 0}\chi\bigl(\, B_r(x)\cap X_{\xi>\xi(x)}\,\bigr),
\]
where $\chi$ denotes the Euler characteristic of a topological space. If $x\in \bsV\setminus X$ we set $i_X(\xi,x)=0$. 

The integer $i_X(\xi,x)$  can be interpreted as a  Morse index   of  the function $-\xi: X\ra \bR$ at $x$; see \cite[\S 9.5]{KaSch} or  Appendix \ref{s: a}. For generic $\xi\in\Sigma\dual$,  we have $i_X(\xi,x)=0$, for all but finitely many points $x\in  X$.

The first goal of this paper  is to  give a very short proof of  the following uniqueness result  closely related   to the uniqueness result of  J. Fu,  \cite[Tm. 3.2]{Fu2}.

\begin{theorem}[Uniqueness] Let $X$  be a compact subanalytic subset   of $\bsV$. Then there exists at most one subanalytic  current $N\in \eC_{n-1}(\Sigma\dual\times\bsV)$  satisfying the following conditions.

\begin{enumerate}

\item The current $N$ is a cycle, i.e., $\pa N=0$.

\item  The current $N$ has compact support.

\item The current $N$ is Legendrian, i.e.,
\[
\lan \alpha\cup \eta,  N\ran=0,\;;\forall \eta\in \Omega^{n-2}\bigl(\, \Sigma\dual\times\bsV\,\bigr).
\]

\item   For any smooth function   $\vfi\in C^\infty(\Sigma\dual\times\bsV)$   we have
\begin{equation}
\lan \vfi dV_{\Sigma\dual}, N\ran=\int_{\Sigma\dual} \Bigl(\sum_{x\in X}  \vfi(\xi,x) i_X(\xi,x)\Bigr)\,dV_{\Sigma^\dual}.
\label{eq: morse-sl}
\end{equation}
\end{enumerate}
\label{th: fu1}
\end{theorem}

\begin{remark} (a) Using \cite[Thm. 4.3.2.(1)]{Feder} we deduce that the  equality   (iv) is equivalent with the  condition
\begin{equation*}
N_\xi= \sum_{x\in X} i_X(\xi,x)\delta_{(\xi,x)},\;\;\mbox{for almost all $\xi\in\Sigma\dual$},
\tag{$\ast$}
\label{tag: ast}
\end{equation*}
where $\delta_{(\xi,x)}$ denotes the  canonical $0$-dimensional current determined by the point $(\xi,x)$. The points $x$ for  which $i(x,\xi)\neq 0$ should be viewed as  critical points of the function $-\xi: X\ra \bR$; see \cite[\S 5.4]{KaSch} or Appendix \ref{s: a}.   Thus, the slice  $N_\xi$   records  both the collection of critical points of $-\xi|_X$ and  their Morse indices.    

\noindent (b) Our sign conventions are different from the ones used in \cite{Ber, Fu2}, but they coincide with the conventions in \cite{CMS}. 
\label{rem: normc}\qed
\end{remark}

When a cycle   as in Theorem \ref{th: fu1} exists, it is called the \emph{normal cycle} of $X$, and we will denote  it by $\bsN^X$.  In the remarkable paper \cite{Fu2}, J. Fu  proved    the following result.

\begin{theorem}[Existence] Every compact subanalytic set $X\subset \bsV$ has a normal cycle $\bsN^X$.\qed
\label{th: fu2}
\end{theorem}

\begin{remark} The  \emph{conormal cycle} $\bsS^X\in \eC_n(T^*\bsV)$  of $X$ constructed by  Kashiwara and Schapira, \cite{KaSch} can be obtained from normal cycle    $\bsN^X$  using  a coning procedure  described explicitly in (\ref{eq: cone}).  The equality (\ref{tag: ast}) is  then a special  of the micro-local index theorem \cite[Thm. 9.5.6]{KaSch}.\qed
\end{remark}

The second goal  of this    paper    is to show that Theorem \ref{th: fu2} is a consequence of Theorem  \ref{th: fu1}.  The  proof    takes full advantage of the subanalytic  context of the problem which prohibits many of the possible pathologies in geometric measure theory. In particular, the  geometric cores   of the arguments are  much more transparent in this context.

\begin{remark} (a)      A  sheaf-theoretic  approach  to the existence of normal cycles based on a conceptually similar approximation method   can be found in \cite[\S 5.2.2]{Schu}. The  concept of  limit of currents is replaced by the concept of specialization, while the uniqueness theorem is replaced   by the injectivity of a certain morphism in Borel-Moore homology, \cite[Eq. (5.19)]{Schu}.  This  injectivity  is ultimately based on a special property of Verdier stratifications, \cite[Cor. 8.3.23]{KaSch}.

 (b)   We    want to point out a key technical difference between the approach in this paper and the approach in \cite{Fu2}.         The uniqueness theorem  \cite[Thm. 3.2]{Fu2}  is formulated in terms of    an integral cycle $\eI(\bsN,\xi,t)\in\eC_{n-1}(\Sigma\dual\times\bsV)$ defined  for \emph{any}  Legendrian cycle $\bsN\in \eC_{n-1}(\Sigma\dual\times\bsV)$ and for almost all $(\xi,t)\in\Sigma\dual\times\bR$; see \cite[Def. 3.1]{Fu2}.    The construction of  the cycle $\eI(\bsN,\xi,t)$ is quite involved, but it has a nice payoff  because it shows that $\eI(\bsN,\xi,t)$  depends continuously on $\bsN$.  This fact     is very useful in      approximation problems.    Our approach  skips  the  construction  of $\eI(\bsN,\xi,t)$, but we have to pay a  small price since the proofs of our convergence results  are a bit more involved.

 (c)    We   believe that  the arguments   used in this paper can  yield an existence result for normal cycles of compact sets in an $o$-minimal   category.   We chose  to work in the subanalytic category   only because of a lack of adequate references for a theory of slicing    of   $o$-minimal currents.     Hardt's subanalytic work \cite{Hardt, Hardt2}   ought to extend with minor  changes to an $o$-minimal context.      \qed
 \label{rem: long}
\end{remark}

Here is a brief outline of the paper. We prove Theorem \ref{th: fu1}   in Section  \ref{s: 3} relying on  an $o$-minimal implementation of  the  strategy in  \cite{Fu1} that affords considerable simplifications.  We construct the normal cycle in Section  \ref{s: 4} via an approximation method. Section \ref{s: 5}   describes how  a combination of  facts proved in this paper and in \cite{Ber} leads to   an alternate proof of the existence part of \cite[Thm. 6.2]{Ber}.  In Section \ref{s: app}  we give an alternate proof  to  the  main convergence   theorem in \cite{Fu3}. 

  We  have included  two  appendices  \ref{s: b}, \ref{s: c}  that survey basic facts about subanalytic sets and currents  used throughout the paper.  In Appendix \ref{s: a}  we   present short  $o$-minimal proofs  of some  basic   facts  of singular Morse theory that we use in  the paper  and  in our opinion  are  not  widely known. In particular, we have included a very short proof of Kashiwara's  non-characteristic deformation lemma in an $o$-minimal setting.


\section{Uniqueness}
\label{s: 3}
\setcounter{equation}{0}

We will  prove Theorem \ref{th: fu1} following a  strategy  that is inspired from \cite{Fu1}.  We first  give a  direct and very short proof  of  a subanalytic version of the general   uniqueness theorem \cite[Thm. 1.1]{Fu1}.  Then, arguing as in the proof of \cite[Thm. 4.1]{Fu1},  we show that   this theorem implies Theorem \ref{th: fu1}.

 \begin{theorem}    Suppose  $S\in \eC_n (\bsV\times \bsV)$   is a subanalytic  current $n$-dimensional current satisfying the following conditions.

 \begin{enumerate}

 \item  The current $S$ is a cycle,  $\pa S=0$,

 \item The current $S$  is lagrangian, i.e., $\omega\cap S=0$.

 \item  The current  $S$ is conical, i.e.,
 \[
 (\mu_\lambda)_* S=S,\;\;\forall \lambda >0,
 \]
 where $\mu_\lambda:\bsV\dual \times \bsV \ra \mu_\lambda:\bsV\dual \times \bsV$ is the rescaling $\mu_\lambda(\xi,x)=(\lambda\xi, x)$.

 \item  If $|S|$ denotes the support of $S$, then the induced map $p:|S|\ra \bsV\dual$ is proper.

 \item     The   set  $p(|S|)$ is a  conical subanalytic subset of $\bsV\dual$ of dimension  $< n=\dim\bsV\dual$.

 \end{enumerate}

 Then $S=0$.
 \label{th: fu11}
 \end{theorem}

 \begin{proof}  We argue  by contradiction. Let $m:=\dim p(|S|)$ so  that $m<n$. If $m=0$, we deduce  that   $|S|\subset \bsV=p^{-1}(0)$.   The \emph{cycle}  $S$ is  subanalytic,   and  of dimension $n= \bsV$ so that   $S=k[\bsV]$ for some integer $k$. On the other hand, conditions (iii) and (iv) imply that  $\pi(|S|)$ is compact, so that $k=0$, and therefore $S=0$ by the constancy theorem.

Suppose $m>0$.   Then we can find a  $m$-dimensional subspace $\bsU\subset \bsV\dual$ so that if $\Phi:\bsV\dual\ra \bsU$ denotes the orthogonal projection onto $\bsU$, then $Z:=\Phi\circ p(|S|)$ is an $m$-dimensional tame subset of $\bsU$.   Moreover, most of the fibers of the  induced map  $\Phi_S:=\Phi|_{p(|S|)}: p(|S|)\ra  Z$ are zero-dimensional. Set $\Psi:=\Phi_S\circ p$; see Figure \ref{fig: 3}.
\begin{figure}[ht]
\centering{\includegraphics[height=2.5in,width=2.5in]{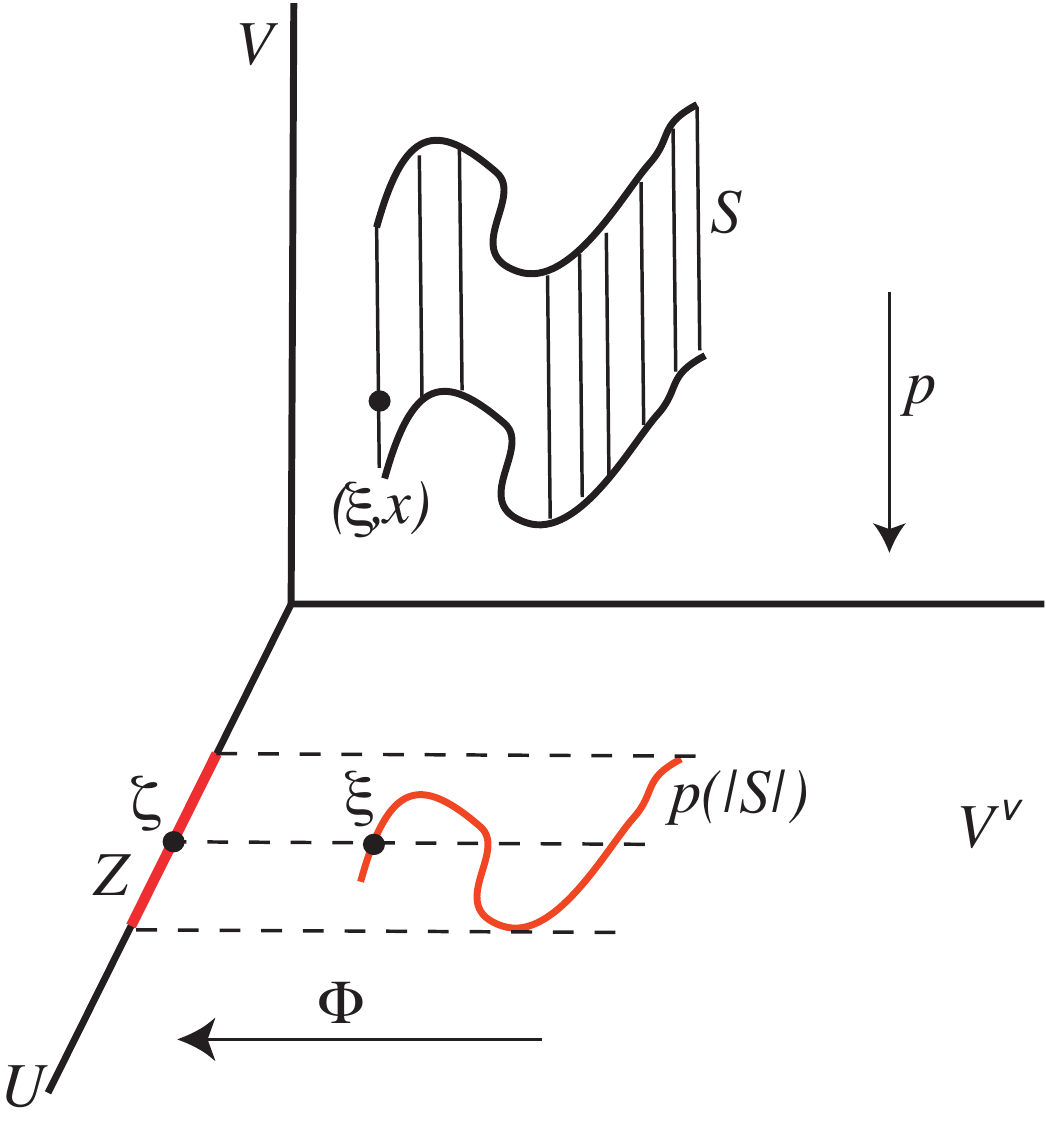}}
\caption{\sl  A rendition of $p(|S|)$. The picture is not entirely accurate since $p(|S|)$ must be a conical subset of $\bsV\dual$.}
\label{fig: 3}
\end{figure}

From the  properties  of  subanalytic sets, or more generally,  sets in an $o$-minimal category, \cite[\S 4]{MvD} or \cite{Dr}, we deduce  that there exists a subanalytic subset $Z'\subset Z$ whose complement has dimension $<m$ such that  following hold.

\begin{itemize}

\item[(c1)] $Z'$ is a $C^2$-manifold.

\item[(c2)]  The induced  map $\Psi: Y':=\Psi^{-1}(Z')\ra   Z'$ is a locally (definably) trivial fibration with $(n-m)$-dimensional fibers. Set $Y'_\zeta:=\Psi^{-1}(\zeta)$, $\zeta\in Z'$.

\item[(c3)] If $w\in Y'$ and near $w$ the set $Y'$ is a $C^2$ manifold, then the differential of $\Psi$ at $w$ is a surjection $\Psi: T_wY'\ra \bsU$.

\item[(c4)] The set $\Xi_{Z'}:=\Phi^{-1}(Z')\cap p(|S|)\subset \bsV\dual$ is a $C^2$-manifold and the induced map $\Xi_{Z'}\stackrel{\Phi}{\ra} Z'$ is a submersion.

\item[(c5)] For any $\zeta\in Z'$ the set $\Xi_\zeta:=\Phi^{-1}(\zeta)\cap p(|S|)$   is finite.\footnote{The definability of the Euler characteristic \cite[\S 4.2]{Dr} implies that   the cardinality of $\Xi_\zeta$ is bounded  from above by a constant independent of $\zeta$.}     In particular, for any $\zeta\in Z'$ the fiber $Y'_\zeta$ is contained in the finite union of planes $\Xi_\zeta\times \bsV$.

\item[(c6)]  For any $\zeta\in Z'$ the slice $\lan S, \Psi, \zeta\ran$ is well defined. It is an $(n-m)$-cycle with support $\cl(Y_\zeta')$.

\end{itemize}

There exists a subanalytic set  $Y''\subset Y'$ of dimension $<n$ such any  $w=\xi\oplus x$ in $Y'\setminus Y''$ belongs  both to the $C^2$-locus of $Y'$,  and to the $C^2$-locus of the fiber   $Y'_{\zeta=\Phi(\xi)}$ that contains  $w$.

 Consider an arbitrary point  $w=\xi\oplus x\in Y'\setminus Y''$  and then choose  a   vector $\dot{w}_1=\dxi_1\oplus \dx_1$  tangent  at $w$   to the fiber $Y'_\zeta$, $\zeta=\Phi(\xi)$.  The condition (c5) shows that  $Y'_\zeta$ is contained in the finite union of planes $\Xi_\zeta\times \bsV$. This implies that  $\dxi_1=0$.

 Using the fact that $S$ is a lagrangian current, i.e.,  $\omega\cap S=0$,  we  deduce that for any $\dot{w}_2=\dxi_2\oplus \dx_2\in T_w Y'$  we have
\[
0=\omega(\dot{w}_1,\dot{w}_2)=\lan \dxi_2,\dx_1\ran - \lan\dxi_1,\dx_2\ran = \lan \dxi_2,\dx_1\ran.
\]
If we denote by $\dx_1^\dag\in\bsV\dual$ the covector dual to  $\dx_1$ we deduce from the above that $\dx_1^\dag$ is perpendicular to $p(T_w |S|)$.   This is an   $m$-dimensional subspace of $\bsV\dual$. At  the point $w=\xi\oplus x$ the  linear map $p: T_wT'\ra T_\xi \Xi_{Z'}$  must be a surjection.     Thus $\dx_1^\dag\perp T_\xi\Xi_{Z'}$.    We deduce that the tangent plane to $Y'_\zeta$ at $\xi\oplus x$     coincides with the plane $\bsT_\xi$,
\[
\bsT_\xi:=\bigl\{ x\in \bsV;\;\; \lan\dxi, x\ran=0,\;\;\forall \dxi\in T_\xi\Xi_{Z'}\,\bigr\}
\]
As we already know, $Y'_\zeta$ is contained in the finite union of planes $\Xi_\zeta\times \bsV$. The above    remarks show that for any $\xi \in \Xi_\zeta$,  and any $C^2$-point    $w$ of the component  of $Y'_\zeta$ contained in $\{\xi\}\times \bsV$  the tangent space $T_wY'_\zeta$ coincides with $\bsT_\xi$.  In other   words  the    Gauss map of the $C^2$-locus of $Y'_\zeta$ has finite range $\bigl\{ \bsT_\xi;\;\;\xi\in\Xi_\zeta\,\bigr\}$.  This shows  that the support of  the slice $\lan S, \Psi, \zeta\ran $ is contained in a finite number of $(n-m)$-dimensional planes.  The slice $\lan S,\Psi,\zeta\ran$   is a $(n-m)$-dimensional cycle  with compact support. The constancy theorem shows  that it must be trivial.   This  implies that $\dim Y_\zeta' <(n-m)$. This contradicts (c2) and thus completes the proof of Theorem \ref{th: fu11}.

\end{proof}

\begin{remark} If we  denote by $d\xi \in \Omega^n(\bsV\dual)$ the Euclidean  volume form on $\bsV\dual$, we see that for a  \emph{subanalytic}   current $S\in \eC_n(\bsV\dual\times \bsV)$ the condition (v) of Theorem \ref{th: fu11} is equivalent  to the condition
\begin{equation*}
(p^*d\xi) \cap  S =0
\tag{$\mathrm{v}'$}
\label{tag: v}
\end{equation*}
employed in \cite[Thm. 1.1]{Fu1}. Indeed, clearly (v) $\Rightarrow$ (\ref{tag: v}).  The implication (\ref{tag: v})  $\Rightarrow$ (v) follows from   Sard's theorem and the  fact that outside a  subanalytic subset  of dimension $\leq (n-1)$ the support $|S|$ can be identified with a real analytic manifold. \qed
\label{rem: v}
\end{remark}

\noindent {\bf Proof  of  Theorem \ref{th:  fu1}.}  Suppose $N_0, N_1\in\eC_{n-1}(\Sigma\dual\times\bsV)$ are two  subanalytic cycles satisfying the condition (i),(ii), (iii), (iv) of the theorem.   Then the   subanalytic cycles   $\pi_* N_i\in \eC_{n-1}(\bsV)$, $i=0,1$, have compact support. Since the reduced homology of $\bsV$ is trivial we deduce  from \cite{Hardt2} that there exist subanalytic currents   $D_i\in \eC_n(\bsV)$ such that
\[
\pa D_i= \pi_*(N_i),\;\;i=0,1.
\]
The constancy theorem (Theorem \ref{th: const}) shows that  the currents $D_i$ are uniquely determined      by the above equality.

 Let $z: \bsV\ra \bsV\dual\times \bsV=T^*\bsV$ denote the zero section of $T^*\bsV$, i.e., $z(x)=(0,x)$,  $\forall x\in\bsV$. Consider the  rescaling map
\[
\mu:[0,\infty)\times  \Sigma\dual \times \bsV\ra \bsV\dual \times \bsV,\;\;  (\lambda, \xi, x)\mapsto (\lambda\xi, x),
\]
and, as in \cite[Prop. 4.8]{Ber}, we form the currents
\begin{equation}
S_i:=\mu_*\bigl(\,[0,\infty)\times N_i\,\bigr) +z_* (D_i),\;\;i=0,1.
\label{eq: cone}
\end{equation}
As explained in \cite[Prop. 4.8]{Ber},  the current $S=S_1-S_0$ satisfies  the assumptions  (i)--(iv) of  Theorem \ref{th: fu11}  and  also the condition (\ref{tag: v}).   Using  the Remark \ref{rem: v} and Theorem \ref{th: fu11} we conclude that  $S=0$.\qed

\begin{remark} Let us observe that   the condition (v) in Theorem  \ref{th: fu11} is  equivalent to the condition that the slices $S_\xi$ are trivial for almost all $\xi$ in $\bsV\dual$.\qed
\label{rem: ber-unique}
\end{remark}

\section{Existence}
\label{s: 4}
\setcounter{equation}{0}

We want to show that Theorem \ref{th: fu1} $\Rightarrow$  Theorem  \ref{th: fu2}. We start by describing  a simple  well known class of compact subanalytic sets  that    admit  normal cycles.

\begin{ex}[Normal cycles of regular domains]   Suppose $X$ is a compact subanalytic  domain in $\bsV$ with $C^2$-boundary. Consider the oriented  Gauss map $\bgamma:\pa X\ra \Sigma\dual$
that  associates to each $x\in \pa X$ the unit covector $\bgamma(x)$ which is dual  to  the  unit \emph{outer}   normal  vector at $x$.    We get an embedding
\[
\Gamma: \pa X\ra \Sigma\dual   \times \bsV,\;\; x\mapsto (\bgamma(x), x),
\]
whose image coincides with the graph of the Gauss map.  Denote by    $[\pa X]$  the integration current defined by $\pa X$ equipped with the  induced\footnote{We use the outer-normal-first convention to orient the boundary.} boundary orientation. Then the cycle  $\Gamma_*[\pa X]$ supported by the graph of the Gauss map  is the normal  cycle of $X$.

Indeed, it   is obviously a subanalytic cycle,  and $\supp \Gamma_*[X]$ is compact  since $X$ is compact.   The Legendrian condition is simply a rephrasing of the fact that for any $x\in \pa X$  the covector $\bgamma(x)$ is conormal to $T_x\pa X$.

To verify (iv) we  first observe that
\[
i_X(\xi,x)=0,\;\;\forall \xi\in \Sigma\dual,\;\;x\in X\setminus \pa X.
\]
  For $x\in \pa X$ denote by $\mathbf{I\!I}_x$ the second fundamental form of  $\pa X$ at $x$.
The equality (\ref{eq: morse-sl}) is a consequence of the following  facts.  Fix  a regular value $\xi$ of $\bgamma$, and a point  $x\in \bgamma^{-1}(\xi)$. Then,

\begin{itemize}

\item [(f1)]  the local  degree of  $\bgamma$ at $x$ is equal to the  sign of the determinant of $-\mathbf{I\!I}_x$;

\item[(f2)] $i(\xi,x)=(-1)^{\nu_+}$, where $\nu_+$ is the number of positive  eigenvalues  of $\mathbf{I\!I}_x$.

\end{itemize}

 For a proof of (f1) we refer to \cite[\S 9.2.3]{N0}.   To prove  (f2) we can assume that  $x=0$, and near $x$ the hypersurface $\pa X$ is   the graph of a quadratic form
 \[
 x^n = q(x^1,\dotsc, x^{n-1})= \sum_{i=1}^{\nu_+} (x^i)^2-\sum_{j=\nu_+ +1}^{n-1} (x^j)^2,
 \]
 while the interior of  $X$ is,  locally,  the region below the graph.   Then $\mathbf{I\!I}_x= q$ and   (f2) now follows  from  standard facts of Morse theory. \qed
\label{ex: reg}
\end{ex}

Suppose now that $X$ is a compact subanalytic set.   Fix an integer $p >2 n=2\dim\bsV$. Then there exists a  subanalytic $C^p$-function $f:\bsV\ra\bR$  such that $f^{-1}(0)=X$; see \cite[Thm. C.11]{MvD}. Set $g:=f^2$ so that $g$ is $C^p$, nonnegative,  subanalytic and $g^{-1}(0)=X$.  The following result  should be obvious.

\begin{lemma}  Fix $R>0$  sufficiently large so  that $X$ is contained in the open ball  $B_R(0)$.   Then there exists $c=c_R>0$ such that,  for any $t\in (0,c_R)$ the level set $g^{-1}(t)$ does not intersect   the sphere $\pa B_R(0)$. \qed
\label{lemm: proper}
\end{lemma}

Denote by $g_R$ the restriction of $g$ to the ball $B_R(0)$ in the above lemma.  The set $\Delta_g$  of critical values   of  $g_R$ is a subanalytic  $0$-dimensional subanalytic  subset of $\bR$  and thus it consists of a finite number of points.

Fix $c_0\in (0,c_R)$   so that the interval $(0,c_0)$ consists only of regular values  of $g_R$.  Then for any $\ve\in (0,c_0)$   the set
\[
X_\ve:=\bigl\{ x\in B_R(0);\;\;g(x)\leq \ve\bigr\}
\]
is a compact subanalytic domain  with $C^2$-boundary.   Therefore it has a  normal cycle $\bsN^\ve=\bsN_{X_\ve}$. The collection $(X_{\ve})_{\ve\in(0,c_0)}$ is an increasing  subanalytic family of compact subanalytic sets such that
\[
X=\bigcap_{0<\ve<c_0} X_\ve.
\]
The collection $(\supp \bsN^\ve)_{\ve\in (0, c_0)}$ is a definable collection of compact, subanalytic manifolds of class $C^p$.  We deduce that their volumes are bounded from above. This shows that the  family of currents $(\bsN^\ve)_{\ve\in (0,c_0)}$ is bounded in the mass norm. The compactness theorem  \cite[Thm. 4.2.17]{Feder} implies  that  there exists a subsequence $\ve_\nu\searrow 0$ such that the  currents    $\bsN^{\ve_\nu}$ converge in  the flat metric to a  integral Legendrian cycle   $\bsN \in \Omega_{n-1}(\Sigma\dual\times\bsV)$. 

To prove that $\bsN$ is a \emph{subanalytic} current  it suffices to show that its support is contained in a subanalytic set of dimension $\leq (n-1)$. To see this we consider the subanalytic set
\[
\eZ=\Bigl\{ (\xi, x)\in\Sigma\dual\times \bsV;\;\exists 0<\ve\leq \frac{c_0}{2}:\;\;x\in\pa X_\ve,\;\;\xi=\bgamma(x)\,\Bigr\}=\bigcup_{0<\ve\leq \frac{c_0}{2}} \supp\bsN_\ve.
\]
Then $\dim \eZ=n$ and $ \cl(\eZ)\setminus \eZ$ is a subanalytic set of dimension $<n$ containing $\supp \bsN$.

  We  want to show that $\bsN$ satisfies  (\ref{eq: morse-sl}). Since
\[
\lan \vfi dV_{\Sigma\dual}, \bsN\ran =\lim_{\nu\ra \infty}\lan\vfi dV_{\Sigma\dual}, \bsN^{\ve_\nu}\ran,\;\;\forall \vfi\in C_0^\infty(\Sigma\dual\times\bsV),
\]
it suffices to show that $\forall \vfi\in C_0^\infty(\Sigma\dual\times\bsV)$ we have
\begin{equation}
\lim_{\ve\searrow 0}\int_{\Sigma\dual}\Bigl(\sum_x \vfi(\xi,x) i_{X_\ve}(\xi,x)\Bigr) dV_{\Sigma\dual}=    \int_{\Sigma\dual}\Bigl(\sum_x \vfi(\xi,x) i_{X}(\xi,x)\Bigr) dV_{\Sigma\dual}.
\label{eq: morse-sl1}
\end{equation}
To do this we will need to use some of the topological facts and terminology presented in Appendix \ref{s: a}.

Observe that $\supp \bsN^\ve$ is a  $C^p$-manifold of dimension $n-1$,  and its projection onto $\bsV$ is $\pa X_\ve$. This shows that the  definable  set
\[
\eN=\bigcup_{0<\ve<c_0}\supp \bsN^\ve.
\]
has dimension $n$. Hence,  there exists a subanalytic set $\Sigma_0\dual  \subset \Sigma\dual$ such that $\dim(\Sigma\dual\setminus\Sigma_0\dual) <\dim\Sigma\dual$ and for any $\xi\in \Sigma_0\dual$ the set $p^{-1}(\xi)\cap \eN$ has  dimension $1$. This implies that for any $\xi\in\Sigma_0\dual$ there exists $c_\xi>0$ such that, and any $\ve\in (0,c_\xi)$  we have
\[
\dim p^{-1}(\xi)\cap\supp \bsN^\ve\leq 0.
\]
Thus, for $\xi\in\Sigma_0\dual$ and  $\ve\in (0, c_\xi)$ the slice $\bsN_\xi^\ve$ is well defined.  Let us point out  that the set of homological critical points of $-\xi: X_\ve\ra \bR$ is  contained in the projection on $\bsV$ of $\supp \bsN_\xi^\ve$.   In particular, if $\xi\in\Sigma_0\dual$, and $0<\ve <c_\xi$ ,   this set of homological critical points is finite.

As explained in Appendix \ref{s: a}, there exists a subanalytic  subset  $\Sigma_1\dual\subset\Sigma\dual$ such that $\dim(\Sigma\dual\setminus \Sigma_1\dual) <\dim\Sigma\dual$ and  for any $\xi\in\Sigma_1\dual$  the set of homological critical points of $-\xi: X\ra \bR$  is finite. Set $\Sigma_*=\Sigma_0\dual\cap \Sigma_1\dual$ so that $\dim \Sigma\dual\setminus \Sigma_* <\dim\Sigma\dual$. We have the following fundamental equality
\begin{equation}
\lim_{\ve\ra 0} \sum_{x\in X} \vfi(\xi,x)\Bigl(\, i_{X_\ve}(\xi,x)-i_{X}(\xi,x)\,\Bigr)=0,\;\;\vfi\in C_0^\infty(\Sigma\dual\times\bsV),\;\;\xi\in  \Sigma_*.
\label{eq: limit}
\end{equation}
The choice $\xi\in\Sigma_*$ guarantees that for any $0<\ve <c_\xi$ the above sum consists of finitely many terms.

Let us first  show that (\ref{eq: limit}) implies (\ref{eq: morse-sl1}).    We will achieve this using the Lebesgue dominated convergence theorem so it suffices to show that there exists  a  constant $C>0$ so that for any $\ve>0$  there exists  a    subanalytic set $\Delta_\ve \subset \Sigma\dual$ such  that $\dim \Delta_\ve <\dim \Sigma\dual$ and
\begin{equation}
\sum_{x\in X_\ve} |i_{X_\ve}(\xi, x)| < C,\;\;\forall  \xi\in \Sigma\dual \setminus \Delta_\ve.
\label{eq: bound}
\end{equation}
For   every $\ve\in (0,c_0)$ we denote by $\Delta_\ve\subset \Sigma\dual$ the set
\[
\Delta_\ve :=\bigl\{\xi\in \Sigma\dual;\;\; p^{-1}(\xi) \;\;\mbox{does \underline{not} intersect  $\supp\bsN^\ve$ transversally}\,\bigr\}.
\]
We deduce that $\dim \Delta_\ve <\dim\Sigma\dual$.      Note that for any $\xi\in\Sigma\dual\setminus S_\ve$ the set $p^{-1}(\xi)\cap \supp \bsN^\ve$ coincides with  $\supp \bsN_\xi^\ve$ and
\begin{equation}
|i_{X_\ve}(\xi,\ve)|=1,\;\;\forall  (\xi, x)\in\supp \bsN_\xi^\ve.
\label{eq:  index}
\end{equation}
The set
\[
\eX :=\bigl\{ (\xi,\ve)\in\Sigma\dual\times (0,c_0);\;\; \xi\in \Sigma\dual\setminus S_\ve\,\bigr\}.
\]
is definable, and for any $(\xi,\ve)\in\eX)$ the    set $\supp\bsN_\xi^\ve$ is finite.    We have thus obtained a  definable collection of finite sets $\bigl(\, \supp N_\xi^\ve\,\bigr)_{(\xi,\ve)\in\eX}$  and  we  conclude that there exists an integer  $K>0$ such that
\begin{equation}
\# \supp\bsN_\xi^\ve <K,\;\;\forall(\xi,\ve)\in\eX.
\label{eq: card}
\end{equation}
Let  $\ve\in(0,c_0)$. Then, for any $\xi\in \Sigma\dual\setminus \Delta_\ve$ we have
\[
\sum_{x\in X}|i_{X_\ve}(\xi,x)|=\sum_{x,\;\;(\xi,x)\in\supp \bsN_\xi^\ve} |i_{X_\ve}(\xi,x)|\stackrel{(\ref{eq: index})}{=} \# \supp\bsN_\xi^\ve \stackrel{(\ref{eq: card})}{\leq} K.
\]
This proves (\ref{eq: bound}).

Let us  prove (\ref{eq: limit}). Fix $\vfi\in C_0^\infty(\Sigma\dual\times\bsV)$ and $\xi\in \Sigma_*$. Define
\[
\Cr_\ve:=\bigl\{ x\in \bsV;\;\; (\xi,x)\in \supp \bsN_\xi^\ve\,\bigr\}.
\]
Using the terminology in Appendix \ref{s: a}, we see that $\Cr_\ve$ is the set of numerically critical points of the function $-\xi$ on $X_\ve$. Consider the $\bR_{\mathrm{an}}$-definable set
\[
\widetilde{\Cr}:=\bigl\{ (t,x)\in [0,c_\xi)\times \bsV;\;\;x\in\Cr_t\,\bigr\},
\]
and denote by $\tau$ the natural projection $\widetilde{\Cr}\ra [0, c_\xi)$. Then there exists  $\delta>0$  such that over the interval  $(0,\delta)$ the map $\tau$ is a locally trivial fibration. Its fibers will consists  of the same number $\ell$ of points.  Thus, we can find  $\bR_{\mathrm{an}}$-definable  continuous  maps $x_1,\dotsc, x_\ell: (0,\delta)\ra \bsV$ such that
\[
\Cr_t=\bigl\{ x_1(t),\dotsc, x_\ell(t)\,\bigr\},\;\;x_i(t)\neq x_j(t),\;\;\forall t\in (0,\delta),\;\;i\neq j.
\]
Set $x_i(0):=\lim_{t\searrow 0}  x_i(t)$. The above limits exist since the functions $x_i(t)$ are definable and $X$ is compact. For $x\in X$ we  define
\[
I_x:=\{ i;\;\; x_i(0)=x\}\subset \{1,\dotsc, \ell\}.
\]
The equality (\ref{eq: limit}) is then a consequence of the following result.

\begin{lemma}
\begin{equation}
i_X(\xi,x)= \lim_{t\ra 0}\sum_{i\in I_x} i_{X_t}(\xi, x_i(t)\,),\;\;\forall x\in X.
\label{eq: limit1}
\end{equation}
In particular, if  $I_x=\emptyset$, then the above limit is zero.
\label{lemma: limit}
\end{lemma}

\begin{proof} Set $c=\xi(x)$.    For $r>0$  and $t\in (0,\delta)$ we set
\[
X_r:= X\cap B_r(x),\;\;X_{t,r}:= X_t\cap B_r(x).
\]
 Choose   $r>0$  and $t_0>0$ sufficiently small such that  the following hold.

\begin{itemize}

\item   The  topological type of  $X_\rho\cap\{\xi>c\}$ is independent  of $\rho\in (0,r]$.

\item  The set  $X_\rho$ is contractible  for any $\rho\leq r$, and the set $X_r\setminus \{x\}$ contains no homological critical points of $-\xi: X\ra \bR$.

\item  $x_i(t)\in B_r(x)$ and $x_j(t)\not\in B_r(x)$   $\forall i\in I_x$, $j\not\in I_x$,  $t\in (0,t_0)$.

\end{itemize}

Fix  $c'<c$ so that the interval $[c', c)$ contains no  critical values of $\xi|_{X_r}$.  We deduce from (\ref{eq: top-drop}) and (\ref{eq: vary})
\[
i_X(\xi,x)=\chi(X_r) -\chi\bigl(X_r\cap\{\xi>c'\}\,\bigr)= \chi(X_r) -\chi\bigl(X_r\cap\{\xi\geq c'\}\,\bigr) .
\]
The only (numerically) critical points of $-\xi$ on $X_{t,r}$ are $\{x_i(t);\;\;i\in I_x\}$.         Choose $t_1<t_0$ such that
\[
\lan \xi,  x_i(t)\ran  > c',\;\;\forall i\in I_x,\;\; 0<t<t_1.
\]
Using  again (\ref{eq: top-drop}) and (\ref{eq: vary}) for the function $-\xi$ on on $X_{t,r}$, $t<t_1$,  we deduce
\[
\sum_{i\in I_x} i_{X_r}(\xi,x)=\chi(X_{t,r})- \chi\bigl(X_{t,r}\cap\{\xi>c'\}\,\bigr)=\chi(X_{t,r})- \chi\bigl(X_{t,r}\cap\{\xi\geq c'\}\,\bigr).
\]
The equality (\ref{eq: limit1})      is obtained by observing that
\[
\lim_{t\searrow 0} \chi(X_{t,r})= \chi(X_r) \;\;\mbox{and} \;\; \lim_{t\searrow 0}\chi\bigl(X_{t,r}\cap\{\xi\geq c'\}\,\bigr)= \chi\bigl(X_r\cap\{\xi\geq c'\}\,\bigr) .
\]
\end{proof}

The  equality (\ref{eq: limit}) is immediate.  Set $\Cr_0:=\{x\in \bsV;\;\;I_x\neq \emptyset\}$.  Then \[
\sum_{x\in \bsV} \vfi(\xi, x) i_X(\xi,x)- \sum_{y\in \bsV}  \vfi(\xi, y) i_{X_\ve}(\xi,y)=\sum_{x\in \Cr_0} \vfi(\xi, x) i_X(\xi,x)- \sum_{y\in \Cr_\ve}  \vfi(\xi, y) i_{X_\ve}(\xi,y)
\]
\[
=\sum_{x\in \Cr_0} \Bigl(\vfi(\xi, x) i_X(\xi,x) -\sum_{i\in I_x} \vfi(\xi, x_i(\ve)) i_{X_\ve}(\xi, x_i(\ve)\,)\Bigr).
\]
The equality (\ref{eq: limit1}) implies that the last term goes to zero  as $\ve\searrow 0$. This proves (\ref{eq: limit}) and (\ref{eq: morse-sl1}).

 From Theorem \ref{th: fu1} we deduce that $\bsN$ must be the normal cycle of $X$ and that $\bsN_{X_\ve}$ converges in the flat metric to $\bsN_X$ as $\ve\searrow 0$.\qed

\section{Normal cycles  of constructible functions}
\label{s: 5}
\setcounter{equation}{0}

Let us explain how the results proved so far lead to an alternate approach to the existence  statement in \cite[Thm. 6.1]{Ber}.  We need to introduce some  terminology.

A  function $ f:\bsV\ra \bZ$ is called constructible if   its range is finite and for any $n\in\bZ$ the  level set $f^{-1}(n)$ is subanalytic.   If we let $I_S$ denote the characteristic  function of a set  $S\subset \bsV$, then  for any constructible function $f$ we can write
\[
f=\sum_{n\in \bZ} nI_{f^{-1}(n)}.
\]
We denote by $\bsC(\bsV)$ the Abelian group of  constructible functions and by  $\bsC_0(\bsV)$ the Abelian group of constructible functions with compact support.   The triangulability theorem \cite[Thm. 8.2.9]{Dr} implies that this group is generated by the characteristic  functions  of  compact subanalytic sets.   

Observe that   for any compact subanalytic sets  $X, Y\subset\bsV$   and any $x\in\bsV$ we have
\[
i_{X\cup Y}(\xi,x) =i_X(X,x)+ i_Y(\xi,x)-i_{X\cap Y}(\xi,x),
\]
for almost all $\xi\in \Sigma\dual$. In view of the uniqueness theorem this implies that
\[
\bsN^{X\cup Y}=\bsN^X+\bsN^Y-\bsN^{X\cap Y}.
\]
From   Groemer's extension theorem \cite[Cor. 2.2.2]{KlR} we deduce  that the correspondence 
\[
\mbox{compact subanalytic subset $X$}\ra \mbox{normal cycle $\bsN_X$}
\]
extends to a group morphism
\[
\bsC_0(\bsV)\ni f\mapsto\bsN^f\in \eC_{n-1}(\Sigma\dual \times \bsV).
\]
The normal cycle $\bsN^f$ of a compactly supported  constructible function $f$ is a compactly supported, subanalytic legendrian cycle.

In the remainder of this  section we want to  compare this construction of the normal cycle of a constructible  function to the one proposed in \cite{Ber}.  

We will need to use the $o$-minimal  Euler characteristic function $\chio$  as defined  in  $o$-minimal   topology; see  Appendix \ref{s: b} and \cite{Dr}.     We will denote by $\chit$ the usual topological Euler characteristic.   These two notions coincide on compact   subanalytic sets, but they  could be quite different on  non-compact ones.   For example, if $B^k$ is an \emph{open} $k$-dimensional ball, then
\[
\chio(B^k)=(-1)^k,\;\; \chit(B^k)=1.
\]
More generally, if $X$ is locally compact, then $\chio(X)$ can be identified with the Euler characteristic of the Borel-Moore  homology of $X$. 

Each  of these two notions of Euler characteristic has its own advantages.  The  $o$-minimal  Euler characteristic $\chio$ is \emph{fully additive} additive, i.e., for \underline{\emph{any}} subanalytic sets $X$ and $Y$ we have
\begin{equation}
\chio(X\cup Y)=\chio(X) +\chio(Y)-\chio(X\cap Y).
\label{eq:  add}
\end{equation}
On the other hand, it is not a  homotopy invariant as its topological cousin $\chit$.

The additivity condition (\ref{eq: add})  and the Groemer extension theorem implies that $\chio$ defines a    linear map $\bsC(\bsV)\ra \bZ$, called the \emph{integral with respect to the Euler characteristic}, and denoted by  $\int d\chi$.

 Suppose  $X$ is a compact subanalytic set, and $\xi\in \Sigma\dual$ is   such that the induced function $-\xi: X\ra \bR$ is  a nice  in the sense of  Appendix \ref{s: a} to which we refer  for notations.  Using  (\ref{eq: top-drop}) we deduce that  for every $t\in \bR$, and every sufficiently small $\ve>0$  we have
 \[
 \sum_{x\in X_{\xi=t}}i_X(\xi,x)  = \chit(X_{\xi\geq t})-\chit(X_{\xi>t})= \chit(X_{\xi\geq t})-\chit(X_{\xi>t-\ve})
 \]
 \[
 = \chit(X_{\xi\geq t})-\chit(X_{\xi\geq t-\ve})= \chio(X_{\xi\geq t})-\chio(X_{\xi\geq t-\ve})
 \]
 \[
 \stackrel{(\ref{eq: add})}{=} \chio(X_{t-\ve< \xi\leq t}) \stackrel{(\ref{eq: add})}{=}  \chio(X_{t-\ve< \xi <t}) +\chio(X_{\xi=t}).
 \]
 For $\ve>0$ sufficiently small  the induced map $\xi:  X_{t-\ve< \xi <t}\ra (t-\ve,t)$ is a locally trivial fibration. We denote its  fiber  by $X_{\xi=t-0}$. Since the  $o$-minimal Euler characteristic of an open interval is $-1$  we deduce
 \[
  \chio(X_{t-\ve< \xi <t}) = \chio\bigl(\, X_{\xi=t-0}\times (t-\ve,t)\,\bigr)= - \chio( X_{\xi=t-0}).
  \]
  We conclude that
  \begin{equation}
 \sum_{x\in X_{\xi=t}}i_X(\xi,x)  = \chio(X_{\xi=t})-\chio(X_{\xi=t-0}).
 \label{eq: jump}
 \end{equation}
 Following  \cite{Ber},  we  associate to  any constructible  function $f\in\eC_0(\bsV)$ and  any $\xi\in \bsV\dual$  the integral $0$-dimensional current    $J_f(\xi)\in\Omega_0(\bR)$ given by
\[
 J_f(\xi):=\sum_{t\in\bR} m_f(\xi, t) \delta_t,
 \]
 where  $m_f(\xi, t)$ is the integer
 \[
 m_f(\xi, t):=\lim_{\ve\searrow 0}\Bigl(\,\int_{\xi=t} f d\chi-\int_{\xi=t-\ve} f d\chi\Bigr)= \lim_{\ve\searrow 0}\int f\cdot \bigl(\,I_{\xi=t}-I_{\xi=t-\ve}\,\bigr) d\chi.
 \]
 Let us point out that when  $f$ is the  characteristic function of a bounded subanalytic set $X$, then
 \[
 m_f(x,t)= \lim_{\ve\searrow 0}\bigr(\, \chio(X_{\xi=t})-\chio(X_{\xi=t-\ve})\,\bigr).
 \]
 If we denote by $\bsI_0(\bR)$ the Abelian group of  integral $0$-dimensional  currents on $\bR$, then    we can organize the above  construction  as  a  \emph{jump} map
 \[
 J_f:\bsV\dual\ra \bsI_0(\bR).
 \]
 This map     is homogeneous  and, as proved\footnote{Our sign conventions are a bit different from the ones in \cite{Ber}. More precisely  if $h_f$ is the support map as defined in \cite{Ber}, then $h_f(\xi)=J_f(-\xi)$. The ``culprit'' for this   discrepancy  is the outer-normal convention used  in Example \ref{ex: reg}.}  in \cite{Ber}, it is also Lipschitz with respect  to   the flat metric on $\bsI_0(\bR)$.  It is also   constructible in the  sense that the function $m_f:\bsV\dual\times \bR\ra \bZ$ is constructible.     Let us observe that  if $f =I_X$, where $X$ compact subanalytic set  with normal cycle  $\bsN^X$, then   the equality  (\ref{eq: jump}) implies that for almost any $\xi\in\Sigma\dual$ we have
 \begin{equation}
 J_{I_X}(\xi)=\beta_*\bsN_\xi^X,
 \label{eq: jump1}
 \end{equation}
 where $\beta :\bsV\dual\times\bsV\ra \bR$ is the bilinear map $\beta(\xi,x)=\xi(x)$.

 We denote by $\eX_\bsV$ the   Abelian group of    constructible, homogeneous, Lipschitz continuous maps $\bsV\dual\ra \bsI_0(\bR)$,  so that   the jump construction   gives  a morphism of Abelian groups 
 \[
 J: \bsC_0(\bsV)\ra \eX_\bsV,\;\;\bsC_0(\bsV)\ni f\mapsto J_f\in \eX_\bsV.
 \]
  According to \cite[Thm. 3.1]{Ber}, this morphism is an \emph{isomorphism} of Abelian groups. The injectivity of $J$ is a consequence of the injectivity of the ``motivic'' Radon transform   of P. Schapira \cite{Scha}. The surjectivity is more subtle, and we refer to \cite{Ber} for details. 
   
  Denote by $\eL_\bsV$ the Abelian group of   subanalytic, conical lagragian cycles $S\in \eC_n(\bsV\dual\times\bsV)$, such that the restriction to $\supp S$  of  the projection $\pi:\bsV\dual\times\bsV \ra \bsV$ is proper.     To  any  compact   subanalytic set $X$ we denote by $\bsS^X\in\eL_\bsV$ the conormal cycle  constructed as in (\ref{eq: cone}),
  \[
  \bsS^X:=\mu_*\bigl(\,[0,\infty)\times \bsN^X\,\bigr) +z_* (X),
  \]
 where $z: \bsV\ra T^*\bsV$ is the zero section, and $\mu: [0,\infty)\times S(T^*\bsV)\ra T^*\bsV$ is the multiplication map
 \[
 [0,\infty)\times S(T^*\bsV)\ni (t,\xi,v)\mapsto (t\xi,v) \in T^*\bsV.
 \]
  The resulting correspondence $X\mapsto\bsS^X$  satisfies the inclusion-exclusion   identity, i.e., 
 \[
 \bsS^{X\cup Y}=\bsS^{X}+\bsS^{Y}-\bsS^{X\cap Y},
 \]
 for any compact subanalytic sets $X, Y$. Invoking the Groemer extension theorem   again we obtain a group morphism $\bsS:\bsC_0(\bsV)\ra \eL_\bsV$. 
 
 Given $S\in \eL_\bsV$, the slice $S_\xi=\lan S,p,\xi\ran$ is defined for almost   every $\xi\in\bsV\dual$.  In \cite[\S 6]{Ber} it is shown that the almost everywhere defined function
 \[
 \bsV\dual\setminus \mbox{negligible set}\ni \xi\mapsto\beta_*S_\xi\in \bsI_0(\bsR)
 \]
 is the restriction    of a function $\si_S\in\eX_\bsV$. We have thus obtained  a morphism of Abelian groups    
 \[
 \si:\eL_\bsV\ra \eX_\bsV,\;\; S\mapsto \si_S,\;\; \si_S(\xi)=\beta_*S_\xi,\;\;\mbox{for almost any $\xi\in\bsV\dual$}.
 \]
 We have the following fundamental result, \cite[p. 403]{Ber}.

  \begin{theorem}  The morphism  $\si$  is injective. More precisely, if $S\in\eL_\bsV$ and 
  \begin{equation*}
\beta_* S_\xi=0,\;\;\mbox{for almost all $\xi\in\bsV\dual$},
\tag{$\mathrm{v}''$}
\label{tag: vp}
\end{equation*}
 then $S=0$.
 \label{th: ber-unique}
 \end{theorem}

 \begin{proof}   For the reader's convenience we decided to include a proof of this result.   What follows is a slightly different   incarnation of the strategy employed in \cite[p.403]{Ber}. We will show that (\ref{tag: vp}) implies that  $S_\xi=0$ for almost any $\xi$.  We then conclude using Remark \ref{rem: ber-unique} and Theorem \ref{th: fu11}.

 First of all, let us observe that since $S$ is conical and $\omega\cap S =0$, then $\alpha\cap S=0$, where we recall that $\alpha\in \Omega^1(T^*\bsV)$ is the canonical $1$-form and $\omega=-d\alpha$ is the canonical symplectic form. 
 
 For any subset $\si\subset \Sigma\dual$ we set
\[
\Xi_\si:=\bigl\{ \xi\in\bsV\dual\setminus 0;\;\;\frac{1}{|\xi|}\xi\in \si\,\bigr\}.
\]
Since $S$ is conical  and subanalytic we can find a definable triangulation $\eK$  of $\Sigma\dual$ such that for any top dimensional (open) face  $\si$ of  there exists

\begin{itemize}

\item  a  finite collection $\eF_\si=\{f_1,\dotsc, f_{\nu(\si)}\}$ of subanalytic,   $C^1$-maps $f: \Xi_\si\ra \bsV$ and

\item  a multiplicity map $m_\si: \eF_\si\ra  \bZ$ 

\end{itemize}
with the following properties

\begin{enumerate}

\item   For any $\xi\in\Xi_\si$, the map $\eF_\si\ni f\mapsto f(\xi)\in \bsV$ is injective.

\item  Any $f\in \eF_\si$ is homogeneous of degree $0$.

\item For any $\xi\in \Xi_\si$ the slice $S_\xi$ is well defined and it is described by
\[
S_\xi=\sum_{f\in \eF_\si} m_\si(f) \delta_{(\xi, f(\xi))}.
\]
\end{enumerate}
Then
\begin{equation}
\beta_* S_\xi=\sum_{f\in\eF_\si} m_\si(f)\delta_{\lan\xi, f(\xi)\ran}=0\in I_0(\bR),\;\;\forall \xi\in \Xi_\si.
\label{eq: zero}
\end{equation}

We fix a top dimensional face $\si$ and we want to prove that $m_\si(f)=0$, $\forall f\in\eF_\si$. We argue by contradiction so we assume that there exists a function $f_0\in \eF_\si$ such  that $m_\si(f_0)\neq 0$.  Then  $\eF_\si'=\eF_\si\setminus \{f_0\}\neq \emptyset$ and  we deduce  from (\ref{eq: zero}) that for any $\xi\in\Xi_\si$ the set
 \[
 G_\xi:=\bigl\{ g\in \eF'_\si;\;\; \lan \xi,g(\xi)\ran =\lan \xi, f_0(\xi)\ran\,\bigr\}
 \]
 is non-empty. The collection $(G_\xi)_{\xi\in\Xi_\si}$  is a definable collection of subsets of the finite set $\eF_\si'$.    From the definable selection theorem we deduce that there exists a definable map
 \[
 \gamma:\Xi_\si\ra \eF'_\si, \;\;\xi\mapsto \gamma_\xi,
 \]
  such that $\gamma_\xi(\xi)\in Z_\xi$, $\forall \xi$. Since $\gamma$ is definable, there exits a definable set $\Delta\subset \Xi_\si$ with the following properties.
  
  \begin{itemize}
  
  \item $\dim\Delta <\dim \Xi_\si=n$.
  
  \item $\Delta $ is closed in $\Xi_\si$.
  
  \item For  any connected component $C$ of $\Xi_\si\setminus \Delta$ the resulting map $\eC\ni \xi\mapsto \gamma_\xi\in \eF_\si'$. is constant. 
  
  \end{itemize}
  
   We will refer to the  connected components of $\Xi_\si\setminus \Delta$ as \emph{chambers}.

 Let  $\xi_0\in\Xi_\si\setminus \Delta$,  and denote  by $\eC_0$ the chamber    containing $\xi_0$. Let $g_0\in \eF_\si'$ be the constant value  of $\gamma\in \eC_0$.      Set
 \[
 u:= f_0(\xi_0)-g_0(\xi_0)\in \bsV,
 \]
  and denote by $u_\dag\in \bsV\dual$ the dual covector.   Since  $\eC_0$ is open we deduce that $\xi_t=\xi_0+ t u_\dag \in \eC_0$
  if $|t|$ is sufficiently small. We deduce
  \[
  \bigl\lan\,\xi_t, f_0(\xi_t)- g_0(\xi_t)\,\bigr\ran=0,\;\;\forall |t|\ll 1.
  \]
 Derivating  the  above  equality  at $t=0$ we deduce
 \begin{equation}
 0=\bigl\lan\, \dxi_0, f_0(\xi_0)-g_0(\xi_0)\,\bigr\ran +\frac{d}{dt}|_{t=0}\bigl\lan\, \xi_0, f_0(\xi_t)\,\bigr\ran-  \frac{d}{dt}|_{t=0}\bigl\lan\, \xi_0, g_0(\xi_t)\,\bigr\ran.
 \label{eq: punch}
 \end{equation}
 The $C^1$-paths $t\mapsto p_t=\xi_t\oplus f_0(\xi_t)\in T^*\bsV$, $t\mapsto q_t=\xi_t\oplus g_0(\xi_t)\in T^*\bsV$ are contained in  the $C^1$-locus of the support of   the current $S$. Since    $\alpha\cap S=0$ we deduce
 \[
 0=\alpha(\dot{p}_0)= \frac{d}{dt}|_{t=0}\lan \xi_0, f_0(\xi_t)\ran,\;\;0=\alpha(\dot{q}_0)=\frac{d}{dt}|_{t=0}\lan \xi_0, g_0(\xi_t)\ran.
 \]
Using this in (\ref{eq: punch}), and observing that $\dot{\xi}_0=u_\dag$, we conclude that $0=\lan u_\dag, u\ran= |u|^2$. Hence $f_0(\xi_0)=g(\xi_0)$. This contradicts the injectivity assumption (i).

 \end{proof}

 Theorem   \ref{th: fu1}  and (\ref{eq: jump1})   imply that we have a commutative diagram
 \[
 \begin{diagram}
 \node{\bsC_0(\bsV)}\arrow{se,b}{J}\arrow[2]{e,t}{\bsS}\node[2]{\eL_\bsV}\arrow{sw,b}{\si}\\
 \node[2]{\bsI_0(\bR)}
 \end{diagram}
 \]
 Since $J$ is an isomorphism, we deduce that the morphism $\si$ must also be surjective.    
 
 \begin{remark} (a)  The  surjectivity of $\si$  is also claimed in \cite[\S 6]{Ber}. However, the proof  is incorrect due to a glitch  in the proof of \cite[Lemma 6.4]{Ber}.  That lemma is an existence  result  having to do with a   certain cycle $S$  with support contained in the closure of a cell $\Gamma$,   and such that $\supp \pa S\subset \pa \Gamma := \cl (\Gamma)\setminus \Gamma$. Loosely,     the lemma states that the relative cycle  determined by $S$ in the homology of the pair $(\cl (\Gamma), \pa \Gamma)$ is trivial.   The  bordism proving the vanishing of this homology class   is a cone on $S$  constructed using  a certain homotopy.  That homotopy is defined only on $\Gamma$, not on $\cl(\Gamma)$, so the homotopy formula  stated as  in \cite[Thm. 4.3]{Ber} cannot be applied.  
 
  (b) From the above  discussion it follows that the morphism $\bsS$ is also bijective. Its  inverse  can  be described in terms of the \emph{local Euler obstruction}, \cite[\S IX.7]{KaSch}, \cite{Scha1}.\qed
  \label{rem: glitch}
 \end{remark}

 \section{Approximation}
\label{s: app}
\setcounter{equation}{0}

We have now developed enough technology to provide an alternate proof  to   the  convergence theorem in \cite{Fu3}.

\begin{theorem}  Suppose  that $(X_k)_{k\geq 0}$ and  $X$ are compact subanalytic  subsets of $\bsV$ satisfying the following  conditions.

\begin{enumerate}

\item[(a)] There exists $R>0$ such that 
\[
X, X_k\subset \bigl\{ x\in \bsV;\;\;|x|\leq R\,\bigr\},\;\;\forall k.
\]
 
\item[(b)]  The sequence of currents $\bsN^k:= \bsN^{X_k}\in \eC_{n-1}(\Sigma\dual\times\bsV)$ is bounded in the mass norm.

\item[(c)] For any $c\in\bR$ and almost any $\xi\in\Sigma\dual$ we have
\[
\lim_{k\ra \infty} \chi\bigl( X_k\cap\{\xi\geq c\}\,\bigr)=  \chi\bigl( X\cap\{\xi\geq c\}\,\bigr).
\]
\end{enumerate}
Then the sequence of currents $\bsN^{X_k}$ converges in the flat metric to $\bsN=\bsN^X$.
\label{th: approx}
\end{theorem}

\begin{proof}  It suffices to show that any subsequence of $(\bsN^k)$  contains  a sub-subsequence that converges in the  flat metric to $\bsN$. To keep the notations at bay, we  will denote by $(\bsN^k)$ the various intervening subsequences   of $(\bsN^k)$.

The conditions (a) and  (b) imply via the compactness theorem for integral currents that $(\bsN^k)$ contains a subsequence convergent in the flat metric to  a subanalytic cycle $\bsN'$.   To prove  that $\bsN'=\bsN$  we will invoke  Theorem \ref{th: ber-unique}  so that we have to show that 
\begin{equation}
\beta_*\bsN_\xi=\beta_*\bsN_\xi',\;\;\mbox{for almost all $\xi\in\Sigma\dual$}.
\label{eq: morse-sl2}
\end{equation}
Denote by $\|-\|_\flat$ the flat norm.  Using the slicing lemma \cite[Lemma 8.1.16]{KP}  we deduce that 
\begin{equation}
\lim_{k\ra \infty}\|\bsN^k_\xi-\bsN'_\xi\|_\flat =0,\;\;\mbox{for almost all $\xi\in\Sigma\dual$}.
\label{eq:  slice-converg}
\end{equation}
Thus, to  prove  (\ref{eq: morse-sl2})  is suffices to show that $\beta_*\bsN_\xi^k$ converges weakly to $\beta_*\bsN_\xi$ for almost any $\xi\in \Sigma\dual$.  

We think of the  integral currents  $\beta_*\bsN_\xi^k$ and  $\beta_*\bsN_\xi$ as  (signed) Borel measures on the real  axis concentrated  on  finite sets. If $I_{[c,c')}$ denotes the characteristic function of the interval $[c,c')$, $c<c'$,  then we deduce from (\ref{eq: vary}) that for almost any $\xi \in \Sigma\dual$  and any $c<c'$ we have
\[
\lan I_{[c,c')}, \beta_*\bsN_\xi\ran=\chi\bigl(\, X\cap\{c'\leq \xi< c\}\,\bigr),\;\;\lan I_{[c,c')}, \beta_*\bsN^k_\xi\ran=\chi\bigl(\, X_k\cap\{\xi\geq c\}\,\bigr),\;\;\forall k\geq 0.
\]
Using (c) and  (\ref{eq: vary}) we deduce 
\[
\lim_{k\ra \infty}\lan I_{[c,\infty)}, \beta_*\bsN^k_\xi\ran= \lan I_{[c,\infty)}, \beta_*\bsN_\xi\ran,\;\;\forall c,\;\;\mbox{for almost any $\xi\in\Sigma\dual$}.
\]
Since $I_{[c,c')}=I_{[c,\infty)}-I_{[c',\infty)}$  we conclude
\begin{equation}
\lim_{k\ra \infty}\lan I_{[c,c')}, \beta_*\bsN^k_\xi\ran= \lan I_{[c,c')}, \beta_*\bsN_\xi\ran,\;\;\forall c<c',\;\;\mbox{for almost any $\xi\in\Sigma\dual$}.
\label{eq: conv-half}
\end{equation}
Let us show that the above equality implies that $\beta_*\bsN^k_\xi$ converges weakly to $\beta_*\bsN_\xi$.

For $k\geq 0$ we define
\[
h_k:\Sigma\dual\ra [0,\infty],\;\; h_k(\xi)=\sum_{x\in X} |i_{X_k}(\xi,x)|.
\]
Observe that if the slice $\bsN^k_\xi$ is defined, then  $h_k(\xi)=\mbox{mass}(\beta_*\bsN_\xi^k)$. If $\mu:= \sup_k \mbox{mass}\,(\bsN^k)$, then
\[
\int_{\Sigma\dual} h_k(\xi)\,|d\xi| \leq \mu,\;\;\forall k.
\]
Hence
\[
{\rm vol}\,\bigl\{ \xi\in\Sigma\dual;\;\;h_k(\xi)>t\,\bigr\}\leq \frac{\mu}{t}.
\]
We set $h_\infty(\xi):=\sup_k h_k(\xi)$ and we deduce from the above inequality that for almost any $\xi\in\Sigma\dual$ we have $h_\infty(\xi)<\infty$.

Let $\vfi \in C_0^\infty(\bR)$.  Denote by $L$ the Lipschitz constant of $\vfi$ and fix  a very small $\ve>0$. Define
\[
\vfi_\ve:\bR\ra \bR,\;\;\vfi(t)=\vfi(n\ve),\;\; \mbox{if} \;\;n\ve\leq t <(n+1)\ve,\;\;n\in\bZ.
\]
Using (\ref{eq: conv-half}) we deduce
\[
\lim_{k\ra \infty}\lan \vfi_\ve, \beta_*\bsN_\xi^k\ran= \lan \vfi_\ve ,\beta_*\bsN_\xi\ran
\]
Choose $k=k(\ve)>0$ such that
\[
\bigl|\, \lan \vfi_\ve, \beta_*\bsN_\xi^k-\beta_*\bsN_\xi\ran\,\bigr|<\ve,\;\;\forall k\geq k(\ve).
\]
Now observe that $\|\vfi-\vfi_\ve\|_\infty \leq L\ve$. We conclude
\[
\bigl|\,\lan \vfi_\ve-\vfi, \beta_*\bsN_\xi\ran\,\bigr|\leq \mbox{mass}\,(\beta_*\bsN_\xi) L\ve,\;\;\bigl|\,\lan \vfi_\ve-\vfi, \beta_*\bsN_\xi^k\ran\,\bigr|\leq   h_\infty(\xi) L\ve,\;\;\forall k.
\]
Hence, 
\[
\bigl|\,\lan \vfi,\beta_*\bsN_\xi^k-\beta_*\bsN_\xi\ran\,\bigr|\leq L\bigl(\, h_\infty(\xi)+\mbox{mass}\,(\beta_*\bsN_\xi)+1\,\bigr)\ve,\;\;\forall k\geq k(\ve).
\]
We deduce that
\[
\lim_{k\ra \infty}\lan \vfi, \beta_*\bsN_\xi^k\ran= \lan \vfi ,\beta_*\bsN_\xi\ran,\;\;\forall\vfi\in C_0^\infty(\bR),
\]
i.e.,  $\beta_*\bsN_\xi^k$ converges weakly to $\beta_*\bsN_\xi$ for almost any $\xi\in \Sigma\dual$.  \end{proof}

\appendix

\section{A fast introduction to $o$-minimal topology}
\label{s: b}

Since the subject of tame  geometry is  not     very familiar to
many  geometers we  devote this section to a   brief  introduction
to this topic.   Unavoidably, we will  have to omit many interesting
details and  contributions, but  we refer to \cite{Co, MvD, Dr}
for  more systematic presentations. For every set $X$ we will denote
by $\eP(X)$ the   collection of all subsets of $X$

An  \emph{$\bR$-structure}\footnote{This is a highly condensed and
special version of  the traditional definition of structure.     The
model theoretic definition  allows for  ordered fields, other than
$\bR$, such as extensions of $\bR$ by ``infinitesimals''. This can
come in handy even if  one is interested  only in the  field $\bR$.}
is a collection $\eS=\bigl\{\, \eS^n\,\bigr\}_{n\geq 1}$,
$\eS^n\subset \eP(\bR^n)$, with the following properties.

\begin{description}

\item[${\bf E}_1.$]   $\eS^n$ contains all the real algebraic subvarieties of $\bR^n$, i.e., the zero sets of  finite collections of polynomial in $n$ real variables.

\item[${\bf E}_2.$]  For every   linear map $L:\bR^n\ra \bR$, the half-plane $\{\vec{x}\in \bR^n;\;\;L(x)\geq 0\}$  belongs to $\eS^n$.

\item[${\bf P}_1.$] For every $n\geq 1$, the family $\eS^n$ is closed under  boolean operations, $\cup$, $\cap$ and complement.

\item[${\bf P}_2.$]  If $A\in \eS^m$, and $B\in \eS^n$, then $A\times B\in \eS^{m+n}$.

\item[${\bf P}_3.$]  If $A\in \eS^m$, and $T:\bR^m\ra \bR^n$ is an affine map, then $T(A)\in \eS^n$.

\end{description}

\begin{ex}[Semialgebraic sets]    Denote by $\bR_{alg}$ the collection of real semialgebraic sets.  Thus,  $A\in \bR^n_{alg}$ if and only if  $A$  is a finite  union of sets,  each of which is described by finitely many polynomial equalities and inequalities. The celebrated Tarski-Seidenberg theorem states that $\eS_{alg}$ is a structure.\qed
\end{ex}

Let $\eS$ be an $\bR$-structure. Then a set that belongs to one of
the  $\eS^n$-s is called   $\eS$-\emph{definable}. If $A, B$ are
$\eS$-definable, then a function $f: A\ra B$ is called
$\eS$-\emph{definable} if its graph $ \Gamma_f :=\bigl\{\, (a,b)\in A\times B;\;\;b=f(a)\,\bigr\}$ is $\eS$-definable.

Given a collection $\eA=(\eA_n)_{n\geq
1}$, $\eA_n\subset\eP(\bR^n)$, we can form a new structure
$\eS(\eA)$, which is the smallest structure containing   $\eS$ and
the sets in $\eA_n$. We say that $\eS(\eA)$ is obtained  from  $\eS$
by \emph{adjoining the collection $\eA$}.

\begin{definition} An $\bR$-structure is called \emph{$o$-minimal} (order minimal) or \emph{tame} if  it satisfies the property

\begin{description}
\item[{\bf T}]    Any set $A\in  \eS^1$ is a \emph{finite} union of open intervals $(a,b)$, $-\infty \leq a <b\leq \infty$, and singletons $\{r\}$. \qed
\end{description}

\end{definition}

\begin{ex} (a) (Tarski-Seidenberg)  The collection  $\bR_{alg}$ of real semialgebraic sets  is a tame structure.

\noindent (b) (A. Gabrielov, R. Hardt, H. Hironaka, \cite{Gab, Hardt2, Hiro})   A \emph{restricted} real
analytic function is a  function $f:\bR^n\ra \bR$ with the property
that there exists a real analytic function $\tilde{f}$ defined in an
open  neighborhood $U$ of the cube $C_n:=[-1,1]^n$ such that
\[
f(x)=\begin{cases}
\tilde{f}(x) & x\in C_n\\
0 & x\in \bR^n\setminus C_n.
\end{cases}
\]
we denote by $\bR_{an}$ the structure obtained from $\eS_{alg}$ by
adjoining the  graphs of all the restricted real analytic functions.
Then $\bR_{an}$ is a tame structure, and the $\bR_{an}$-definable
sets are called \emph{(globally) subanalytic sets}. \qed
\end{ex}

The  definable sets  and function of a tame structure have  rather remarkable \emph{tame} behavior which prohibits  many pathologies.  It is perhaps   instructive to give an example of function which is not definable in any tame structure. For example, the function $x\mapsto \sin x$ is not definable in a tame structure because the intersection of its graph with the horizontal axis is the  countable set $\pi\bZ$  which  violates  the tameness condition ${\bf T}$.

 We will list below some of the nice properties of the sets and function definable  in a  fixed tame structure  $\eS$.  Their proofs can be found in \cite{Co, Dr}. We will interchangeably refer to  sets or functions definable in  a given tame structure $\eS$ as \emph{definable}, \emph{constructible} or  \emph{tame}.

\smallskip

\noindent \ding{227}   (\emph{Piecewise  smoothness of  tame
functions.})  Suppose $A$ is a definable  set, $p$ is a
positive integer, and $f: A\ra \bR$ is a definable function. Then
$A$ can be partitioned into finitely many   definable sets
$S_1,\dotsc, S_k$,     such that each  $S_i$ is a $C^p$-manifold,
and each of the restrictions $f|_{S_i}$ is a $C^p$-function.

\noindent  \ding{227} (\emph{Triangulability.})  For every   compact
definable set $A$, and any finite collection of definable  subsets
$\{S_1,\dotsc, S_k\}$, there exists  a compact simplicial complex
$K$, and a  definable homeomorphism $\Phi: |K|\ra A$ such that  all the sets $\Phi^{-1}(S_i)$ are unions of  relative
interiors of faces of $K$.

\noindent \ding{227} (\emph{Dimension.})  The  dimension of a
definable  set $A\subset \bR^n$ is the supremum over all the
nonnegative integers $d$ such that there exists a $C^1$  submanifold
of $\bR^n$ of dimension $d$ contained in $A$.  Then $\dim A
<\infty$, and $\dim (\cl(A)\setminus A) <\dim A$.

\noindent \ding{227}(\emph{Definable selection.}) Any tame map  $f: A\ra B$ (not necessarily continuous) admits a tame section, i.e.,  a tame map $s: B\ra A$ such that $s(b)\in f^{-1}(b)$, $\forall b\in B$.

\noindent \ding{227} (\emph{Local triviality of tame maps}) If $f: A\ra B$ is a tame continuous map, then there exists a  tame triangulation of $B$ such that over the relative interior of any face the map  $f$ is  a locally trivial fibration.

\noindent \ding{227} (\emph{The $o$-minimal Euler characteristic}) There exists a function $\chio:\eS\ra \bZ$ uniquely characterized  by the following  conditions.

\begin{itemize}

\item $\chio(X\cup Y)=\chio(X)+\chio(Y)-\chio(X\cap Y)$, $\forall X,Y\in \eS$.

\item  If $X\in \eS$ is compact, then $\chio(X)$ is the usual Euler characteristic of $X$.

\end{itemize}

\noindent \ding{227} (\emph{The scissor principle})  Suppose $A$ and $B$ are two tame sets. Then  the following are  equivalent

\begin{itemize}

\item  The sets $A$ and $B$ have the same  $o$-minimal Euler characteristic and dimension.

\item There exists  a tame bijection $f: A\ra B$. (The map $f$ \emph{need  not be continuous}.)

\end{itemize}

\noindent \ding{227} (\emph{Crofton formula}, \cite{BK}, \cite[Thm.
2.10.15, 3.2.26]{Feder}.) Suppose $E$ is an Euclidean space, and
denote by $\Graff^k(E)$ the Grassmannian of affine subspaces of
codimension $k$ in $E$.  Fix an invariant measure $\mu$ on
$\Graff^k(E)$.\footnote{ The measure $\mu$ is unique up to a
multiplicative constant.} Denote by $\eH^k$ the $k$-dimensional
Hausdorff measure. Then there exists a constant $C>0$, depending
only on $\mu$, such that for every compact, $k$-dimensional  tame
subset $S\subset E$ we have
\[
\eH^k(S)= C\int_{\Graff^k(E)} \chi(L\cap S) d\mu(L).
\]

\noindent \ding{227} (\emph{Finite volume.}) Any compact
$k$-dimensional tame set has finite $k$-dimensional Hausdorff
measure. 

\noindent \ding{227} (\emph{Uniform volume bounds.}) If $f:A\ra B$ is a proper, continuous  definable map such that all the fibers  have dimensions $\leq k$, then   there exists $C>0$ such that
\[
 \eH^k\bigl(\, f^{-1}(b)\,\bigr) <C,\;\;\forall b\in B.
\]
\qed

\section{Subanalytic  currents}
\label{s: c}
\setcounter{equation}{0}

In this appendix  we gather  without proofs a few facts  about  the subanalytic currents introduced by R. Hardt in \cite{Hardt, Hardt2}.  Our  terminology    concerning currents closely  follows that of Federer \cite{Feder} (see also the more accessible \cite{KP, Mor}).  However, we changed some  notations to better resemble notations used in  algebraic topology.

Suppose $X$ is a $C^2$, oriented Riemann manifold of dimension $n$.
We denote by $\Omega_k(X)$ the space of $k$-dimensional currents in
$X$, i.e., the topological dual space of the space
$\Omega^k_{cpt}(X)$ of smooth, compactly supported $k$-forms on $X$.
We will denote by
\[
\lan\bullet,\bullet\ran: \Omega^k_{cpt}(X)\times \Omega_k(X)\ra \bR
\]
the natural pairing.  The boundary of a current $T\in \Omega_k(X)$
is  the $(k-1)$-current defined via the Stokes formula
\[
\lan \alpha, \pa T\ran :=\lan d\alpha, T\ran,\;\;\forall \alpha\in
\Omega^{k-1}_{cpt}(X).
\]
For every  $\alpha\in \Omega^k (X)$,  $T\in \Omega_m(X)$,  $k\leq m$
define $\alpha \cap T\in \Omega_{m-k}(X)$ by
\[
\lan\beta , \alpha \cap T  \ran =\lan \alpha\wedge \beta, T
\ran,\;\;\forall \beta\in \Omega^{n-m+k}_{cpt}(X).
\]
We have
\[
\lan \beta, \pa (\alpha \cap T)\ran = \lan \,d\beta, (\alpha\cap
T),\ran = \lan \alpha\wedge d\beta, T\ran
\]
\[
=(-1)^k \lan d(\alpha\wedge \beta) -d\alpha\wedge \beta, T\ran =
(-1)^k \lan \beta, \alpha \cap\pa T\ran +(-1)^{k+1} \lan \beta,
d\alpha \cap T\ran
\]
which yields the \emph{homotopy formula}
\begin{equation}
\pa (\alpha\cap T)= (-1)^{\deg \alpha} \bigl(\, \alpha \cap \pa
T-(d\alpha) \cap T\,\bigr). \label{eq: homotop}
\end{equation}
We have the  following important result\, \cite[\S 4.1.7]{Feder}.
\begin{theorem}[Constancy theorem]   Suppose $S\in \Omega_k(\bsV\dual\times \bsV)$ is a $k$-dimensional \emph{cycle}  whose support is contained in a $k$-dimensional  affine subspace $\bsU\subset \bsV\dual\times\bsV$.    Then there exists  an orientation $\mathbf{or}$ on $\bsU$ and an integer $\ell$ such that  $S=\ell[\bsU,\mathbf{or}]$ where $[\bsU,\mathbf{or}]$ is the current of integration along the oriented  affine plane $\bsU$.  In particular, if $\supp S$ is compact, then $S=0$.\qed
\label{th: const}
 \end{theorem}

We say that a set  $S\subset \bR^n$ is \emph{locally subanalytic}
if for any $p\in \bR^n$ we can find an open ball $B$ centered at
$p$ such that $B\cap S$ is globally subanalytic.

\begin{remark} There is a rather subtle distinction between globally subanalytic and locally subanalytic sets. For example, the graph of the function $y=\sin(x)$ is a locally subanalytic subset of $\bR^2$, but it is not a globally subanalytic  set. Note that a compact, locally subanalytic set is globally subanalytic.\qed
\end{remark}

If $S\subset \bR^n$ is an orientable, locally subanalytic, $C^1$
submanifold of $\bR^n$ of dimension $k$,  then any orientation
$\ori_S$ on $S$ determines a  $k$-dimensional current $[S,\ori_S]$
via the equality
\[
\lan \alpha, [S, \ori_S]\ran:=\int_S \alpha,\;\;\forall \alpha\in
\Omega^k_{cpt}(\bR^n).
\]
The integral in the right-hand side is well defined because any
bounded, $k$-dimensional  globally subanalytic set has finite
$k$-dimensional Hausdorff  measure.  For any open, locally
subanalytic   subset $U\subset \bR^n$ we  denote by $[S,\ori_S]\cap
U$  the   current $[S\cap U, \ori_S]$.

For any  locally subanalytic subset $X\subset \bR^n$ we denote by
$\eC_k(X)$ the    Abelian subgroup of $\Omega_k(\bR^n)$
generated  by currents of the form $[S,\ori_S]$,  as above, where
$\cl(S)\subset X$. The above operation $[S,\ori_S]\cap U$, $U$ open
subanalytic extends to a morphism of  Abelian groups
\[
\eC_k(X)\ni T\mapsto T\cap U\in\eC_k(X\cap U).
\]
We will refer to the elements of $\eC_k(X)$ as \emph{subanalytic
(integral) $k$-chains} in $X$.

Given compact subanalytic sets $A\subset X\subset \bR^n$ we set
\[
\eZ_k(X,A)=\bigl\{ T\in \eC_k(\bR^n);\;\;\supp T\subset X,\;\;\supp
\pa T\subset A\,\bigr\},
\]
and
\[
\eB_k(X, A)=\bigl\{ \pa T + S;\;\;T\in \eZ_{k+1}(X,A)), \;\;S\in
\eZ_k(A)\,\bigr\}.
\]
We set
\[
\eH_k(X,A):=\eZ_k(X,A)/\eB_k(X,A).
\]
R. Hardt has proved in \cite{Hardt2} that the assignment
\[
(X,A)\longmapsto \eH_\bullet(X,A)
\]
satisfies the Eilenberg-Steenrod   homology axioms with
$\bZ$-coefficients. This implies   that $\eH_\bullet(X,A)$
is naturally isomorphic  with the integral homology  of the pair.

To describe the intersection theory of subanalytic chains  we need
to recall a  fundamental result of R. Hardt, \cite[Theorem
4.3]{Hardt}.  Suppose $E_0, E_1$ are two oriented  real Euclidean
spaces of dimensions $n_0$ and respectively $n_1$,  $f:E_0\ra E_1$
is a real analytic map, and $T\in \eC_{n_0-c}(E_0)$ a subanalytic
current of codimension $c$.  If $y$ is a regular value of $f$,  then
the fiber $f^{-1}(y)$  is  a submanifold  equipped with a natural
coorientation and thus defines a subanalytic  current $[f^{-1}(y)]$
in $E_0$  of codimension $n_1$, i.e., $[f^{-1}(y)]]\in
\eC_{d_0-d_1}(E_0)$. We would like to define the intersection of $T$
and $[f^{-1}(y)]$  as a subanalytic current  $\lan T, f, y\ran\in \eC_{n_0-c-n_1}(E_0)$.     It  turns out that  this
is possibly quite often, even in cases when $y$ is  not a regular
value.

\begin{theorem}[Slicing Theorem]  Let $E_0$, $E_1$, $T$ and $f$ be  as above, denote by $dV_{E_1}$ the Euclidean volume form on $E_1$,   by $\bom_{n_1}$ the volume of the unit ball in $E_1$, and set
\[
\eR_f(T):=\bigl\{ y\in E_1;\;\codim (\supp T )\cap f^{-1}(y) \geq
c+ n_1,\;\codim (\supp \pa T )\cap f^{-1}(y)\geq
c+n_1+1\,\bigr\}.
\]
For  every $\ve>0$ and $y\in E_1$ we define $T\bullet_\ve
f^{-1}(y)\in \Omega_{n_0-c-n_1}(E_0)$ by
\[
\bigl\lan\, \alpha, T\bullet_\ve f^{-1}(y)\,\bigr\ran
:=\frac{1}{\bom_{n_1}\ve^{n_1}}\bigl\lan\, (f^*dV_{E_1})\wedge\alpha ,
T\cap \bigl(\,f^{-1}(B_\ve(y)\,\bigr)\,\bigr\ran,\;\;\forall
\alpha\in \Omega^{n_0-c-n_1}_{cpt}(E_0).
\]
Then  for every $y\in \eR_f(T)$, the currents $T\bullet_\ve
f^{-1}(y)$ converge weakly as $\ve>0$   to a  subanalytic  current
$\lan T, f,y\ran\in \eC_{n_0-c-n_1}(E_0)$ called  the
\emph{$f$-slice} of $T$ over $y$. Moreover,   the map
\[
\eR_f\ni y\mapsto \lan T, f, y\ran\in \eC_{d_0-c-d_1}(\bR^n)
\]
is continuous in the   locally   flat  topology.\qed
\label{th: slice}
\end{theorem}

\section{Elementary Morse theory on  singular spaces}
\label{s: a}
\setcounter{equation}{0}

Throughout this appendix we fix an $o$-minimal category of sets and we will refer to the sets and maps  in this category as \emph{tame} or \emph{definable}. For a  topological space $Z$ we denote by $H^\bullet(Z)$  the (\v{C}ech) cohomology with real coefficients, and we define its topological Euler characteristic to be the integer
\[
\chit(Z)=\sum_{k\geq 0} \dim H^k(Z),
\]
whenever the sum in the right-hand side is well defined.

Suppose $X$ is a  locally closed  tame subset of $\bsV$, and $S$ is a closed tame subset of $X$. We define the local cohomology of $X$ along $S$ (with real coefficients) to be
\[
H^\bullet_S(X):=H^\bullet(X,X\setminus S).
\]
We can now define the \emph{local cohomology  sheaves} $\h^\bullet_S=\h^\bullet_{X/S}$ to be the sheaves associated to the presheaves $U\longmapsto H^\bullet_{S\cap U}(U)$.

If $x\in X$ and $U_n(x)$ denotes the  open ball of radius $1/n$ centered at $x$ then  for every $m\leq n$ we have morphisms $H^\bullet_{S\cap U_m}(U_m)\ra H^\bullet_{S\cap U_n}(U_n)$, and then the stalk of $\h^p_S$ at $x$ is the inductive limit $\h^\bullet_S(x):=\lim_{n\ra \infty} H^\bullet_{S\cap U_n}(U_n)$. Observe that since $X$ is locally  conical we have
\begin{equation}
\h^\bullet_S(x)=0\mbox{ for every}\;\; x\in (X\setminus S).
\label{eq: vanish}
\end{equation}
We set
\[
\chi_S(X):=\sum_k(-1)^k \dim H^k_S(X),\;\;\chi_S(x):=\sum_{k\geq 0}(-1)^k\dim \h^k_S(x).
\]
We have a Grothendieck  spectral sequence  converging to $H^\bullet_S(X)$ whose $E_2$ term is $E_2^{p,q}= H^p(X, \h^q_S)$.

 If  it happens that the local cohomology sheaves are supported by finite sets then
\[
H^{p,q}(X,\h^q_S)=0,\;\;\forall p>0,
\]
so that the spectral sequence degenerates  at the $E_2$-terms.  In this case we have
\begin{equation}
H^q_S(X)\cong H^0(X,\h^q_S)\cong \bigoplus_{x\in X}\h^q_S(x).
\label{eq: local-coho}
\end{equation}
In particular
\begin{equation}
\chi_S(X)=\sum_{x\in X} \chi_S(x).
\label{eq: local-chi}
\end{equation}

Suppose now that $X$ is a compact  connected   tame subset of $\bsV$ and $f: X\ra \bR$ is a  definable, continuous function.    We will write
\[
X_{f\geq c}:= \{x\in X;\;\;f(x)\geq c\},\;\;   X_{f\leq c}:= \{x\in X;\;\;f(x)\leq c\},\;\;\mbox{etc.}
\]
 A real number  $c$ is said to be a  \emph{regular} value of $f$ if there  exists $\ve >0$ such that the induced map
\[
f :  \{|f-c|<\ve\}\subset X\ra (c-\ve, c+\ve)
\]
is a locally, definably trivial fibration.   A real number $c$ is said to be a \emph{critical value} of $f$ if  it is not  a critical value.    The local trivialization theorem in \cite[Chap. 9]{Dr} implies that the  set of critical values of $f$ is finite.

 Fix a real number $c$,  and consider   the    sheaves of local cohomology     $\h^\bullet_{f\geq c}$ associated to the closed  subset set $S=X_{f\geq c}\subset X$.  Note that   if $f(x) =c$, and $c$ is a regular value,  then $\h^\bullet_{f\geq c}(x)=0$.    Any number  outside  the range  of $f$  is automatically a regular value.

 A point $x\in X$ is called a \emph{homological} critical point    of  $f$ if
 \[
 \h^\bullet_{f\geq c}(x)\neq 0,\;\;\mbox{where}\;\;c=f(x).
 \]
 Note that in this case $c$ must be a  critical value of $f$. We denote by $\Cr_f$ the  set  of homological critical points of $f$.  We say that $f$ is \emph{nice} if     $\Cr_f$ is finite.

   We  define the \emph{index} of $f$ at  a point $x\in \bsV$ to be the integer
 \[
 \bm(f,x)=\bm_X(f,x):=\chi(\h^\bullet_{f\geq c}(x)) =\lim_{r\searrow 0} \chi\Bigl(\, H^\bullet\bigl(\, B_r(x)\cap X, B_r(x)\cap X_{f<c}\,\bigr)\,\Bigr),
 \]
 where $c=f(x)$ and  $B_r(x)$ denotes the \emph{open} ball in $\bsV$ of radius $r$, centered at $x$. Note that $\bm(f,x)=0$ if $x\not\in X$. Due to the local conical structure of $X$ we have
 \begin{equation}
 \bm(f,x)= 1- \lim_{r\searrow 0} \chi\bigl(\, H^\bullet(\,B_r(x)\cap X_{f<c}\,)\,\bigr).
 \label{eq: local-morse}
 \end{equation}
 A point $x\in X$ is called a \emph{numerically critical point} of $f$ if $\bm(f,x)\neq 0$. We denote by $\Cr_f^\#$ the set of  numerically critical points of $f$. Observe that $\Cr_f^\#\subset \Cr_f$.

One can ask the following natural question. Given a compact  tame set $X$, do there exist  nice continuous tame functions $f: X\ra \bR$? The answer is, yes, plenty of them.   More precisely one can show (see \cite{GM, Pig}) that there exists  a subset $\Delta_X\subset \Sigma\dual$, such that $\dim\Delta_X<\dim\Sigma\dual=(n-1)$ and for any $\xi\in\Sigma\dual\setminus \Delta_X$, the induced function $\xi:X\ra \bR$ has only a finite number of homological critical points.

 \begin{ex} (a) Note that if $X$ is a  compact  $C^2$-submanifold of $\bsV$,  $f$  is a Morse function on $X$, and $x$ is a critical point  of $f$ with Morse index $\lambda$, then $\bm(f,x)=(-1)^\lambda$. In this case $\Cr_f=\Cr_f^\#$.

(b) If $X$  is a compact, convex,  subanalytic subset of $\bsV$ and $\xi:\bsV\ra \bR$ is a linear map, then a point$x\in X$ is critical for the restriction of $\xi$ to $X$ if and only if $x$ is a minimum point for $\xi$. In  this case we have $\bm(\xi,x)=1$.\qed
 \end{ex}

\begin{lemma}[Kashiwara]   Suppose  that $X$  is a  compact  connected  tame subset of $\bsV$ and  $f: X\ra \bR$ is a  nice  continuous tame  function that contains no  homological critical points on the level set $\{f=c\}$. Then the  inclusion induced morphism $H^\bullet(X_{f< c+\ve} )\ra H^\bullet(X_{f<c})$ is an isomorphism for all $\ve>0$ sufficiently small.
\label{lemma: kashi}
\end{lemma}

\begin{proof}    We first prove that  the morphism $H^\bullet(X_{f\leq  c} )\ra H^\bullet(X_{f<c})$. is an isomorphism.   Indeed, it suffices to show  that   the local cohomology of $Z=X_{f\leq c}$ along $S=\{f\geq c\}$ is trivial.     This follows from (\ref{eq: local-coho}) by observing that the local cohomology sheaves  $\h^\bullet_{Z/S}(x)$ are trivial.

Next observe that for some $\ve_0>0$   the induced map  $f:\{c<f<c+\ve_0\}\ra (c,c+\ve_0)$
is a locally trivial fibration. This implies that for any $\ve'<\ve'' <\ve_0$ the induced morphism
\[
H^\bullet(X_{f< c+\ve''} )\ra H^\bullet(X_{f< c+\ve''} )
\]
is an isomorphism.  We conclude by observing that $H^\bullet(X_{f\leq  c} )=\limi_{t\searrow 0} H^\bullet(X_{f<c+t})$. \end{proof}

\begin{remark}  Kashiwara's lemma is valid   in a much more general context, \cite[Prop. 2.7.2]{KaSch}.   The proof  in the general case  is much more involved.\qed
\end{remark}

 Suppose  is a compact  connected   tame set subset of $\bsV$ and  $f: X\ra \bR$ is a  nice,  continuous  tame  function.  Fix $c\in f(X)$. We have  $H^\bullet_{X_{f\geq c}}(X_{f< c+\ve})=H^\bullet(X_{f< c+\ve}, X_{f<c})$. From the equality
\[
\bigcap_{\ve>0} X_{f< c+\ve}= X_{f\leq c}
\]
we deduce
\begin{equation}
\limi_\ve H^\bullet_{X_{f\geq c}}(X_{f< c+\ve})=H^\bullet_{X_{f\geq c}}(X_{f\leq c})= H^\bullet_{X_{f= c}}(X_{f\leq c})= H^\bullet(X_{f\leq c}, X_{f<c}).
\label{eq: inj}
\end{equation}
Using (\ref{eq: vanish}) and   (\ref{eq: local-chi})  we deduce  that for any  sufficiently small $\ve>0$  we have
\begin{equation}
\sum_{\substack{x\in \Cr^\#_f,\\f(x)=c}} \bm(f,x)=\chit\bigl(\,X_{f\leq c}\,\bigr)-\chit\bigl(\,X_{f<c}\bigr)= \chit\bigl(\,X_{f< c+\ve}\,\bigr)-\chit\bigl(\,X_{f<c}\bigr).
\label{eq:   top-drop}
\end{equation}
 Suppose   now that  $c',c\in f(X)$, $c'<c$ and the interval $(c,'c)$ contains  no  critical values of $f$.   Then,
 \begin{equation}
 \chit(X_{f<c'})=\chit(X_{f<c}).
 \label{eq: =}
 \end{equation}
 Iterating (\ref{eq: top-drop}) and (\ref{eq: =})  we deduce  that for any $c,c'\in f(X)$,   $c'<c$, we have
 \begin{equation}
  \chit( X_{f\leq c})-\chit(X_{f \leq c'})= \sum_{\substack{x\in \Cr_f^\#\\ c'<  f(x)  \leq  c}} \bm(f,x).
  \label{eq: vary}
  \end{equation}


\end{document}